\documentclass[reqno]{amsart}
\usepackage[utf8]{inputenc}
\usepackage{amsmath,amsfonts,amssymb,amsthm,dsfont,bm}
\usepackage{enumitem}
\usepackage[usenames,dvipsnames]{xcolor}
\usepackage[colorlinks=true]{hyperref} 
\usepackage{verbatim}
\usepackage{graphicx,overpic}
\usepackage{mathabx}
\usepackage{epsfig}
\usepackage[outdir=./]{epstopdf}
\usepackage{scalerel}

\hypersetup{linkcolor=Violet,citecolor=PineGreen}
\newtheorem{thm}{Theorem}[section]
\newtheorem{lem}[thm]{Lemma}
\newtheorem{prop}[thm]{Proposition}

\theoremstyle{definition}
\newtheorem{defn}[thm]{Definition}

\newtheorem{rmk}[thm]{Remark}
\numberwithin{equation}{section}

\newcommand{\txtd}{\textnormal{d}}

\newcommand{\N}{\mathbb{N}}
\newcommand{\R}{\mathbb{R}}

\newcommand{\U}{\mathcal{U}}
\newcommand{\cK}{\mathcal{K}}
\newcommand{\cI}{\mathcal{I}}

\newcommand{\w}{\omega}

\newcommand{\ep}{\varepsilon}

\newcommand{\nin}{{n\in\N}}
\newcommand{\nti}{{n\to\infty}}

\makeatletter
\newcommand{\setword}[2]{%
  \phantomsection
  #1\def\@currentlabel{\unexpanded{#1}}\label{#2}%
}
\makeatother

\begin{document}
\title[Rate-induced tipping and asymptotic series]
{Estimating rate-induced tipping\\ via asymptotic series and a Melnikov-like method}
\author[C.~Kuehn]{Christian Kuehn}
\author[I.P.~Longo]{Iacopo P. Longo}
\email[Christian Kuehn]{ckuehn@ma.tum.de}
\email[Iacopo P. Longo]{longoi@ma.tum.de}

\address[C.~Kuehn, I.P.~Longo]{Technical University of Munich, 
Department of Mathematics,
Boltzmannstr.~3,
85748 Garching bei M\"unchen, Germany.}

\thanks{CK was partly supported by a Lichtenberg Professorship of the VolkswagenStiftung. CK also acknowledges partial support of the EU within the TiPES project funded the European Unions Horizon 2020 research and innovation programme under grant agreement No. 820970. IPL was partly supported by MICIIN/FEDER project RTI2018-096523-B-100 and by European Union’s Horizon 2020 research and innovation programme under the Marie Skłodowska-Curie grant agreement No 754462. CK and IPL also want to thank Hildeberto Jardon-Kojakhmetov for interesting discussions regarding asymptotic analysis in the early phase of this project.}

\date{}
\begin{abstract}
The paper deals with the study of rate-induced tipping in asymptotically autonomous scalar ordinary differential equations. We prove that, in such a tipping scenario, a solution which limits at a hyperbolic stable equilibrium of the past limit-problem loses uniform asymptotic stability and coincides with a solution which limits at a hyperbolic unstable equilibrium of the future limit-problem. We use asymptotic series to approximate such pairs of solutions and characterize the occurrence of a rate-induced tipping by using only solutions calculable on finite time intervals. Moreover, we show that a Melnikov-inspired method employing the asymptotic series allows to asymptotically approximate the tipping point.

\end{abstract}
\maketitle
\section{Introduction}
\label{secintro}

The study of critical transitions in complex systems has been object of a considerable scientific attention and effort. The reason lies in the need of more reliable mathematical tools to describe a variety of tipping phenomena in climate systems~\cite{paper:LHK, paper:WALC}, financial markets~\cite{paper:MLS,paper:YSY}, neuroscience~\cite{paper:Meisel,paper:Hill}, ecological and natural systems~\cite{paper:OKW,scheffer, paper:SvNHH} among others~\cite{scheffer,paper:KCT2}. In recent years~\cite{paper:WALC}, the phenomenon of rate-induced tipping has been proposed to be an alternative mechanism for a critical transition, with respect to the more classical autonomous bifurcations and noise-induced tipping~\cite{paper:APW}. Rate-induced tipping can be seen as a special type of a non-autonomous bifurcation, which manifests itself on a finite time interval, on which the parameters change significantly and non-adiabatically. This encompasses various real scenarios for example in ecology~\cite{paper:OKW,paper:SvNHH, paper:VWF}, climate~\cite{paper:AAJ,paper:AWVC,paper:LD, paper:WALC}, biology~\cite{paper:Hill} and quantum mechanics~\cite{paper:Kato}. Although a considerable number of applied examples for rate-induced tipping have been documented, the theoretical foundations are currently still lagging a bit behind this development. In this work, we contribute to build better theoretical methods based upon non-autonomous nonlinear dynamics and bifurcation theory to advance the study of transitions with critical variation rates of a parameter.  
 \par\smallskip
 
In the original formulation of rate-induced tipping, Ashwin et al.~\cite{paper:APW} consider a family of topologically equivalent autonomous differential problems parametrized in $\lambda\in [\lambda_-,\lambda_+]\subset \R$.
\[ 
\frac{\txtd x}{\txtd t}=\dot x =f(x,\lambda),\qquad x=x(t)\in\R^N,
\]
where $f$ is a sufficiently regular function from $\R^N\times \R$ onto $\R^N$ and for any fixed value of $\lambda$, the qualitative behavior of the solutions of any of such systems is supposed to be completely determined by its critical points which are assumed to be hyperbolic. Then, a smooth function $\Lambda:\R\to[\lambda_-,\lambda_+]$ is considered such that $\Lambda$ is asymptotically constant and in particular it converges to $\lambda_-$ as $t\to-\infty$ and to $\lambda_+$ as $t\to\infty$.  Therefore, one obtains a non-autonomous dynamical system of the form 
\begin{equation}\label{eq:init-probA}
\dot x =f\big(x,\Lambda(rt)\big),\qquad x\in\R^N,t\in\R.
\end{equation}
which is asymptotically autonomous in the past and in the future. The parameter $r>0$ represents the rate at which the time-dependent sweep between $\lambda_-$ and $\lambda_+$ takes place. System~\eqref{eq:init-probA} is associated to a set of continuous functions which play an important role in the study of rate-induced tipping. On the one hand, we have the continuous families of quasi-static equilibria  that map any $t\in\R$ to the hyperbolic equilibria of the autonomous problem obtained for $\lambda=\Lambda(rt)$. The graphs of these functions represent the adiabatic displacement of the equilibria upon the variation of $\lambda\in [\lambda_-,\lambda_+]$. On the other hand, it has been shown that for each stable hyperbolic equilibrium $X^s_-$ of the past limit-problem,~\eqref{eq:init-probA} has a solution $x^r_-:(-\infty,\beta)\to\R^N$ limiting to $X^s_-$ as $t\to-\infty$ which is also locally pullback attracting~\cite{paper:APW}. \par\smallskip 

Pullback attractivity is an inherent concept of non-autonomous dynamical systems and entails a property of attraction in the past.
Depending on $r$, the qualitative behaviour of any of such locally pullback attracting solutions may change  considerably. If these locally pullback attractive trajectories are all defined on the whole real line and limit at the ``respective equilibria'' of the autonomous future limit-problem as $t\to+\infty$, the system is said to end-point track the curves of quasi-static equilibria. When this is not true, a rate-induced tipping is said to happen. In this case, a local pullback attracting solution can become unbounded in finite time or it may converge to an equilibrium of the future limit-problem which does not represent the expected landing equilibrium determined by an adiabatic change of the parameter. 
\par\smallskip

The phenomenon of rate-induced tipping has been identified and studied analytically and numerically in several formulations. For example, in
one-dimensional systems (as in~\cite{paper:APW}), higher-dimensional systems (as in~\cite{paper:AA,paper:WXJ,paper:X,paper:KJ}), discrete systems (in~\cite{paper:Kiers}), multiscale systems (as in~\cite{paper:PW,paper:AWVC}), deterministic non-autonomous systems (in~\cite{paper:LNOR}), set-valued dynamical systems (as in~\cite{paper:Ca}),  and random dynamical systems (as in~\cite{paper:Hartl}).
\par\smallskip

In this paper, we study the occurrence of rate-induced tipping in scalar non-linear ordinary differential equations, i.e.~systems like~\eqref{eq:init-probA} where $N=1$, both from an analytical and a geometrical point of view. In particular, we highlight the following achievements.

\begin{itemize}[leftmargin=*, itemsep=2pt]
\item We show that for scalar asymptotically autonomous differential problems, the occurrence of a rate-induced tipping coincides with the loss of uniform asymptotic stability by one of the locally pullback attracting solutions limiting at the stable equilibria of the past limit-problem.  We like to point out that a loss of hyperbolicity---a stronger form  of uniform asymptotic stability which involves an exponential rate of convergence---has been proved for a certain class of scalar \emph{quadratic} differential equations with possibly nonautonomous asymptotic dynamics~\cite{paper:LNOR}. Although such a stronger result seems still hard to prove for general scalar problems, we believe that the achievement in our work contributes to reinforce the relation between rate-induced tipping and nonautonomous bifurcation theory.

\item The locally pullback solutions can be approximated by asymptotic series expansions whose terms can be calculated using only values of $f$, $\Lambda$ and their derivatives, and the families of quasi-static equilibria, which are all a-priori-known quantities of the given problem. By the term asymptotic, we mean that the approximations are reliable for $r>0$ small enough. Nevertheless, we show that for every $r>0$ the approximations are always reliable on  suitable half-lines of the real line.

\item The estimate on the errors obtained via the asymptotic series expansions permits to identify a sufficient condition for end-point tracking of a locally pullback attractive solution associated to a curve of quasi-static equilibria.

\item  The asymptotic approximations can be used to characterize the occurrence of rate-induced tipping via the change of relative order between pairs of solutions suitably chosen in a neighborhood of the locally pullback solutions. 
\end{itemize}

In fact, the rate-induced tipping mechanism studied here is reminiscent of global autonomous bifurcation problems, where the relative order of certain invariant manifolds is tracked via Melnikov/Lin-type methods. In summary, our combination of analytical and geometric methods shows that rate-induced tipping in the scalar context can be viewed via non-autonomous bifurcations due to its coincidence with the loss of uniform asymptotic stability, and it can hence be analyzed using a combination of classical tools from invariant manifold theory and asymptotic analysis. Furthermore, our results show that asymptotic series expansions provide an interesting additional tool for rate-induced tipping problems to connect the very slow (quasi-static or adiabatic) regime to the rate-induced tipping regime.\par\smallskip

The paper is organized as follows.
In section~\ref{sec:non-aut}, we set the notation and recall some notions on non-autonomous ordinary differential equations flows and attractivity. Section~\ref{sec:prelim} contains the assumptions, definitions and preliminary results on rate-induced tipping for scalar differential equations. In particular, we recall that every curve of quasi-static stable equilibria is associated to a locally pullback attracting solution and show that, similarly, every  branch of quasi-static unstable equilibria containing an unstable hyperbolic equilibrium of the future limit-problem  is associated to a locally pullback repelling solution. The main result of this section is Theorem~\ref{thm:fundamentals} where we prove that a rate-induced tipping coincides with a loss of uniform asymptotic stability by a locally pullback attracting solution limiting at a stable equilibrium of the past limit-problem through a collision with a locally pullback repelling solution limiting at an unstable equilibrium of the future limit-problem. Section~\ref{sec:asympt_series_exp} deals with the asymptotic approximation of the locally pullback solutions of~\eqref{eq:init-probA}. In Proposition~\ref{prop:expansions}, we show that every locally pullback solution (either attracting or repelling) can be approximated via an asymptotic series expansion when $r$ is sufficiently small. In particular, the zero-order approximation coincides with the respective curve of quasi-static equilibria and the coefficients of higher order can be calculated using only values of $f$, $\Lambda$ and their derivatives. Notably, all the coefficients of order higher or equal than one tend to zero as $|t|\to\infty$. This guarantees that for every $r>0$ the asymptotic series approximation is always effective on a suitable half-line of the real line. This deduction allows us to give information on the occurrence or the absence of a rate-induced tipping. Particularly, Proposition~\ref{prop:suff-end-point} contains a sufficient condition for end-point tracking upon assuming that the asymptotic series expansion holds on the whole real line with a suitably bounded (but not necessarily small) error. On the other hand, in Theorem~\ref{thm:melnikov} we characterize the occurrence of rate-induced tipping by choosing suitable pairs of solutions such that one is close to the pullback attracting solution in the sufficiently far but finite past, and the other is close to the pullback  repelling solution in the sufficiently far but finite future. We show that a change of relative order between such pairs of solution characterizes the occurrence of a rate-induced tipping for the system.

\section{Background: Non-autonomous ODEs, flows and attractivity}
\label{sec:non-aut}

In this section, we introduce the general notation used in the work and recall some basic notions and definitions about dynamical systems induced by non-autonomous ordinary differential equations. We will restrict the presentation to the definitions and results which are relevant for this work. We refer the reader to~\cite{book:GS} for an in-depth presentation.\par\smallskip

For any $N\in\N$ natural number, we shall denote by $\R^N$ the $N$-dimensional euclidean space with its norm $|\cdot|$. The symbols $\R^+$ and $\R^-$ will denote the intervals $[0,\infty)$ and $(-\infty,0]$, respectively. For every $V\subset\R^N$, $W\subset\R^M$, and $\nin$, we denote by $\mathcal{C}^n(V,W) $  the set of functions $f:V\to W$ whose partial derivatives up to order $n$ exist and are continuous. In particular,  the set of continuous functions from $\R\times\R$ into $\R$, denoted by $\mathcal{C}=\mathcal{C}(\R\times\R,\R) $ is endowed  with the compact-open topology. This means that a sequence $(f_n)_\nin$ in $\mathcal{C} $ converges to a function $f\in \mathcal{C}$ if and only if for every set $\cK\times \cI$, where $\cK$ and $\cI$ are compact subsets of $\R$, the sequence $(f_n)_\nin$ converges uniformly to $f$ on $\cK\times \cI$. It is well-known that one can define a (global) continuous flow $\pi: \mathcal{C}\times \R\to \mathcal{C}$ (that is $\pi$ satisfies the identity and group properties) on $\mathcal{C}$ via the shift map  $(f,\tau)\mapsto \pi(f,t)=f_\tau$ where $f_\tau$ is defined by $f_\tau(x, t)=f(x,t+\tau)$ (see~\cite{book:GS} for example). For any $f\in \mathcal{C}$, we denote by $\w(f)$ the \emph{omega limit-set} of $f$, i.e.~ the set of functions $g\in \mathcal{C}$ such that there is a sequence $ (t_n)_\nin$ in $\R$, $t_n\to\infty$, for which $g=\lim_\nti f_{t_n}$. In particular, $f$ is called \emph{positively precompact} if for all sequences $ (t_n)_\nin$ in $\R$, $t_n\to\infty$ there is a subsequence $ (t_{n_j})_{j\in\N}$ and a $g\in \mathcal{C}$ such that $\lim_{j\to\infty}f_ {t_{n_j}}=g$. The \emph{alpha limit set} $\alpha(f)$ of $f$ is defined analogously by considering the limit functions for sequences $ (t_n)_\nin$ in $\R$ with $t_n\to-\infty$. Now, fix $f\in \mathcal{C}$  and consider the non-autonomous differential equation
\begin{equation}
\label{eq:non-aut}
\dot x= f(x,t),
\end{equation}
We say that~\eqref{eq:non-aut} is \emph{forward asymptotically autonomous} (resp.~\emph{backward asymptotically autonomous}) if there is $g\in \mathcal{C}$ such that $\omega(f)=\{g\}$ (resp.~$\alpha(f)=\{g\}$). We say that~\eqref{eq:non-aut} is \emph{asymptotically autonomous} if it is both forward and backward asymptotically autonomous.\par\smallskip

We call $f$ \emph{admissible} if~\eqref{eq:non-aut} has a unique solution for all initial data $(t_0,x_0)\in\R^2$ and denote by $x(t,t_0,x_0)$ its unique solution  satisfying $x(t_0,t_0,x_0)=x_0$, where $t$ belongs to the maximal interval of definition $\cI_{t_0,x_0}$. A special notation will be used for the solutions with initial data $(0,x_0)$, $x_0\in\R$: $x(t,0,x_0)=x(t,f,x_0)$ and their maximal interval of definition $\cI_{x_0,f}$. Notice, in particular, that given any initial data $(t_0,x_0)\in\R^2$, the solution $x(\cdot,t_0,x_0)$ of~\eqref{eq:non-aut}, with $x(t_0)=x_0$, and the~solution $x(\cdot, f_{t_0}, x_0)$, of $\dot x=f_{t_0}(x,t)$, $x(0)=x_0$, are the same up to a time translation. Specifically, one has that 
\[
x(t,t_0,x_0)=x(t-t_0,f_{t_0},x_0),\quad\text{ for all } t\in \cI_{t_0,x_0}.
\]
\par\smallskip

We recall the following fundamental result of continuity.

\begin{lem}[Kamke]\label{lem:kamke}
Let $\mathcal{A}$ denote the collection of all admissible functions $f$ in $\mathcal{C}(\R^N\times\R,\R^N)$. Then, the solution function $x(t,f,x_0)$ is continuous on the subset $ \R\times\mathcal{A}\times \R^N$ for which it is defined. That is, if $(f_n)_\nin$ is a sequence of admissible functions in $\mathcal{C}(\R^N\times\R,\R^N)$ with limit $f$ in $\mathcal{C}(\R^N\times\R,\R^N)$, where $f$ is admissible, if $(x_n)_\nin$ is a sequence in $\R^N$ with limit $x_0$ and if $(t_n)_\nin$ is a sequence in $\cI_{x_n, f_n}$ with limit $t$, then $t\in \cI_{x_0, f}$ and 
\[
\lim_\nti x(t_n,f_n,x_n)=x(t,f,x_0).
\]
\end{lem}

We also recall some standard definitions of stability.
\begin{defn}
A solution $\widetilde x:[t_0,\infty)\to\R$ of $\dot x=f(x,t)$ with $f: \R\times \R \to \R$, is called \emph{uniformly stable} if for any $\ep>0$ there is $\delta=\delta(\ep)>0$ such that  $|\widetilde x(s) - x_0|<\delta$ for some $s>t_0$ implies $|\widetilde x(t) - x(t,s,x_0)|<\ep$ for all $t>s$.  

A solution $\widetilde x:[t_0,\infty)\to\R$ of $\dot x=f(x,t)$ with $f: \R\times \R \to \R$ is called \emph{uniformly asymptotically stable} if it is uniformly stable and there is $b>0$ such that for every $\ep>0$ a $T(\ep)>0$ exists such that if $|\widetilde x(s) - x_0|<b$, for some $s>t_0$, then $|\widetilde x(t) - x(t,s,x_0)|<\ep$ for all $t>s+T(\ep)$.
\end{defn}

The following classical result (see~\cite[Theorem F]{paper:Art}) relates the uniform asymptotic stability of the solutions of $\dot x=f(x,t)$ to the dynamic behavior of the asymptotic equations.

\begin{prop}\label{prop:asympt-stability}
Let $f\in \mathcal{C}$ be positively precompact and such that all the elements of $\omega(f)$ are admissible. Moreover, assume that $f(0,t)=0$ for all $t\in\R$ and $\widetilde x(t)=0$, $t\in\R$ is the unique solution of $\dot x= f(x,t)$ with $x(0)=0$. Then, $\widetilde x(t)$ is uniformly asymptotically stable if and only if there is a neighborhood $W$ of $0$ which is a region of uniform attraction collectively with respect to the family of limiting equations of $\dot x= f(x,t)$. That is, for any compact set $\cK\in W$ and any $\ep>0$ there is a $T>0$ such that whenever any solution $\overline x(t)$ of $\dot x =g(x,t)$, $g\in\omega(f)$, satisfies $\overline x(t_0)\in \cK$, then $|\overline x(t)|<\ep$ for all $t\ge t_0+T$.
\end{prop}

\begin{rmk}
The results of the previous proposition apply to any arbitrary solution $\widetilde x$ of $\dot x = f(x,t)$ by considering the system
\[
\dot x =F(x,t)= f\big(x+\widetilde x(t),t\big)- f\big(\widetilde x(t),t\big).
\]
\end{rmk}

Besides attraction in the future, non-autonomous systems admit a form of attraction in the past which takes the name of pullback attractivity. This notion plays an important role in the study of rate-induced tipping. The following definition of local
pullback attractivity and repulsivity is taken from~\cite{paper:LNOR} and it was in turn adapted from~\cite[Section~2.3]{rasmln}. It is worth noting that it is weaker than the definition provided in~\cite{paper:APW}.

\begin{defn}\label{defn:pullback}
A solution $\overline x\colon(-\infty,\beta)\to\R$ (with $\beta\le\infty$) of
\eqref{eq:non-aut} is called {\em locally pullback attracting\/} if there exist
$s_0<\beta$ and $\delta>0$ such that if $s\le s_0$ and $|x_0- \overline x(s)|<\delta$
then $x(t,s,x_0)$ is defined for $t\in[s,s_0]$, and in addition
\[
 \lim_{s\to-\infty}|\overline x(t)-x(t,s,x_0)|=0
 \quad \text{for all $t\le s_0$}.
\]

A solution $\overline x\colon(\alpha,\infty)\to\R$  (with $\alpha\ge\infty$) of
\eqref{eq:non-aut} is called {\em locally pullback repelling\/}  if there exist
$s_0>\alpha$ and $\delta>0$ such that if $s\ge s_0$ and $|x_0- \overline x(s)|<\delta$
then $x(t,s,x_0)$ is defined for $t\in[s_0,s]$, and in addition
\[
 \lim_{s\to\infty}|\overline x(t)-x(t,s,x_0)|=0
 \quad \text{for all $t\ge s_0$}\,.
\]
\end{defn}

\section{General setting, assumptions and preliminary results on rate-induced tipping}\label{sec:prelim}
In this section, we introduce the notation regarding rate-induced tipping and clarify the overall setting of the work. Moreover, we recall a few key results from~\cite{paper:APW} and complete them (where necessary) in accordance with our framework. The section contains two new main results: Proposition~\ref{prop:suff-end-point} and Theorem~\ref{thm:fundamentals}.  

Proposition~\ref{prop:suff-end-point} provides a sufficient condition for the absence of a rate-induced tipping.  Unlike $\ep$-close tracking or forward basin stability \cite{paper:APW}, it does not require knowledge on the whole history of the pullback attractor or of the basins of attraction for the quasi-static equilibria. It is shown that proximity to $X^s_+$ at only one value of time beyond a certain threshold is sufficient. This is not surprising, given the persistence of hyperbolic solutions, but it seems to be absent from the literature on rate-induced tipping---for this reason we declare it as a proposition rather than a theorem. 

Theorem~\ref{thm:fundamentals} characterizes a rate-induced tipping with the loss of uniform asymptotic stability of an attractive orbit which collides with a repelling solution at  the tipping point. A stronger result has been proved for a certain class of scalar quadratic differential equations---possibly with nonautonomous asymptotic dynamics---showing that such collision entails a loss of hyperbolicity \cite{paper:LNOR}. In this context, hyperbolicity means the existence of an exponential dichotomy for the associated variational equation \cite{book:KR}, generalizing the notion of exponential rate of asymptotic convergence for solutions which are not stationary or periodic.
\par\smallskip

We will work under the following fundamental assumptions.\par\smallskip

\setword{\upshape(\textbf{H0})}{H0} Consider $f\in \mathcal{C}^2(\R\times\R,\R)$ bounded and  so that the parametric differential problems  
\begin{equation}\label{eq:parametric}
\dot x =f(x,\lambda),\qquad \lambda\in [\lambda_-,\lambda_+],
\end{equation}
are well-defined, admit existence and uniqueness of the solutions and their continuous dependence with respect to  parameter and initial data. Moreover, assume that for every $ \lambda\in [\lambda_-,\lambda_+]$, except at most a finite number of points in $(\lambda_-,\lambda_+)$ the autonomous dynamical system induced by~\eqref{eq:parametric} has only hyperbolic fixed points. In particular, we shall assume that there is at least one stable hyperbolic fixed point $X^s_-$ of the problem $\dot x =f(x,\lambda_-)$ such that, upon the variation on $ \lambda$, $X^s_-$ is continuously perturbed into a stable hyperbolic fixed point $X^s_\lambda$ of $\dot x =f(x,\lambda)$, with $ \lambda\in [\lambda_-,\lambda_+]$. In particular, we simplify the notation at $\lambda=\lambda_+$ by writing $X^{s}_+:=X^{s}_{\lambda_+}$. We shall call the continuous function $X^s:[\lambda_-,\lambda_+]\to\R$ defined by 
\[
\lambda\mapsto X^s(\lambda)=X^s_\lambda,
\]
the family of \emph{quasi-static stable equilibria} associated to $X^s_-$.
\par\smallskip
Now, let $\Lambda\in \mathcal{C}^2(\R, [\lambda_-,\lambda_+])$, be any bounded and strictly increasing function such that 
\[
\lim_{t\to\pm\infty}\Lambda(t)=\lambda_\pm\quad\text{and}\quad \lim_{t\to\pm\infty}\Lambda'(t)=0,
\]
and for any $r>0$ consider the non-autonomous differential problem
\begin{equation}\label{eq:init-prob}
\dot x =f\big(x,\Lambda(rt)\big),\qquad x\in\R,t\in\R.
\end{equation}
Under the considered assumptions, the solution function $x\colon\U\subset\R^3\to\R$, $(t,s,x_0)\mapsto x(t,s,x_0)$ is continuous. Notice also that 
\[
\lim_{t\to-\infty}f\big(x,\Lambda(rt)\big)= f(x,\lambda_-)=:f_-(x)\ \ \text{and}\ \ \lim_{t\to\infty}f\big(x,\Lambda(rt)\big)= f(x,\lambda_+)=:f_+(x),
\] 
uniformly for $x$ in a compact subset of $\R^N$.
Therefore,~\eqref{eq:init-prob} is asymptotically autonomous. We will call the autonomous differential problem $\dot x= f_-(x)$  the past limit-problem, and $\dot x= f_+(x)$ the future limit-problem of~\eqref{eq:init-prob}, and denote by $x(\cdot, f_-,x_0)$ and $x(\cdot, f_+,x_0)$ the unique solution of $\dot x= f_-(x)$ and $\dot x= f_+(x)$, respectively, with initial data  $(0,x_0)$.

In particular, we will also have the time-parametrized curve of quasi-static equilibria, which (with a little abuse of notation) we keep denoting with the same symbol, defined by 
\[
\R^+\times\R\ni(r,t)\mapsto X^{s}\big(\Lambda(rt)\big)=:X^{s}(rt)\in\R.
\]
Due to construction, $X^{s}$ is continuous and bounded. We shall also assume that there is $p>0$ such that $\partial_xf\big(X^{s}(t),\Lambda(t)\big)<-p<0$, for all $t\in\R$, where $\partial_xf$ is the first derivative of $f$ with respect to the variable $x$. \par\smallskip

The assumptions considered so far are inherited from~\cite{paper:APW}. As we will see in the very next results, they allow us to construct (and extend) the same framework of~\cite{paper:APW} and analyze rate-induced tipping in terms of the behaviour of a locally pullback attracting solution. Importantly, the existence of this solution and the persistence of the connection between the attractor in the past and the one in the future for "small" rates, are obtained  through arguments of singular perturbation theory which require these assumptions (see Propositions \ref{prop:pbk-attr-rep} and \ref{prop:suff-end-point}). It is therefore natural that such set of assumptions appears in other works on the subject (e.g.~\cite{paper:KJ}). We also like to point out that milder assumptions on regularity and a different argument using persistence of hyperbolic solutions are used in \cite{paper:LNOR} to obtain similar results for a class of nonautonomous quadratic differential problems---in that context the geometric argument is not applicable. In this work, however, regularity will not be considered as an issue. In fact, in Section \ref{sec:asympt_series_exp} we will always require that our function $f$ admits partial derivatives of any order with respect to the variable $x$ and that they are continuous. 
It is also worth noting that, in other contexts, for example where rate-induced tipping is treated in terms of geometric thresholds \cite{paper:X} or investigated numerically \cite{paper:HNRS}, the assumption on  regularity for $f$ and $\Lambda$ is generally reduced to $C^1(\R\times\R,\R)$.  
\par\smallskip

The following proposition recalls a crucial result in~\cite{paper:APW} where the solution stemming from the hyperbolic equilibrium $X^s_-$ of the past limit-problem is recognized as locally pullback attractive. We include a similar statement for a locally pullback repelling solution stemming from a hyperbolic equilibrium $X^u_+$ of the future limit-problem.
\begin{prop}
\label{prop:pbk-attr-rep}
Consider the assumption~\ref{H0}. Then, the following statements hold true.
\begin{itemize}[leftmargin=*,itemsep=2pt]
\item For every $r>0$ there is a solution $x^r_-(\cdot)=x_-(r,\cdot)\in \mathcal{C}((-\infty,\beta^r_-),\R^N)$  of~\eqref{eq:init-prob} such that 
\[
\lim_{t\to-\infty}x^r_-(t)=X^s_-,
\]
and $x^r_-$ is the only solution satisfying such limit.
Moreover, $x^r_-(\cdot)$ is locally pullback attractive and there exist $T>0$ and a neighbourhood $U\subset\R$ of $X^s_-$ such that $x^r_-(\cdot)$ is the only trajectory remaining in $U$ for all $t<-T$. 

\item If the problem $\dot x =f(x,\lambda_+)$ has also an unstable hyperbolic equilibrium $X^u_+$, then  for every $r>0$ there is a solution $x^r_+(\cdot)=x_+(r,\cdot)\in \mathcal{C}((\beta^r_+,\infty),\R^N)$ of~\eqref{eq:init-prob} such that
\[
\lim_{t\to\infty}x^r_+(t)=X^u_+\,,
\]
and $x^r_+$ is the only solution satisfying such limit.
Moreover, $x^r_+(\cdot)$ is locally pullback repelling and there exist $T>0$ and a neighbourhood  $V\subset\R$ of $X^u_+$ such that  $x^r_+(\cdot)$ is the only trajectory remaining in $V$ for all $t>T$.
\end{itemize}
\end{prop}
\begin{proof}
The proof of the statement for $x^r_-$ is given in Theorem~2.2 of~\cite{paper:APW}. In order to prove the statement for $x^r_+$ firstly notice that since  $\dot x =f(x,\lambda_+)$ has an unstable hyperbolic equilibrium $X^u_+$, then there is $\widetilde\lambda\in[\lambda_-,\lambda_+)$ such that $X^u_+$ is (continuously) perturbed into an unstable hyperbolic equilibrium $X^u(\lambda)$ for all $\lambda\in [\widetilde\lambda,\lambda_+]$. Denote by $\widetilde\tau$ the real number such that $\Lambda(\widetilde\tau)=\widetilde\lambda$ and consider the function $\widetilde\Lambda:\R\to[\widetilde\lambda,\lambda_+]$ defined by
\[
\widetilde\Lambda(\tau)=\begin{cases}
\widetilde\lambda, &\text{if }\tau<\widetilde\tau,\\
\Lambda(\tau),&\text{if }\tau\ge \widetilde\tau.
\end{cases}
\]
Notice that $\widetilde\Lambda$ is continuous but not differentiable for $\tau=\widetilde\tau$ in general. However, given $\ep>0$, one can always construct a function $\overline \Lambda \in \mathcal{C}^2(\R,[\widetilde\lambda,\lambda_+])$ which coincides with $\widetilde\Lambda$ outside a ball of radius $\ep$ centered at $\widetilde \tau$  as a convolution with a mollifier (see~\cite{book:AFP}), i.e. 
\[
\overline \Lambda(\tau)=\int_\R\widetilde\Lambda(s)\rho_\ep(\tau-s)\,\txtd s,
\]
where for example $\rho_\ep(\tau)=\exp\big(1/(|\ep\tau|^2-1)\big)$ if $|\tau|<\ep$ and $\rho_\ep(\tau)=0$ otherwise. Then, we obtain the non-autonomous problem 
\begin{equation}\label{eq:25/10-17:11}
\dot x =f\big(x,\overline\Lambda(rt)\big),
\end{equation}
and notice that, for any $t>(\widetilde\tau+\ep)/r$, a solution of~\eqref{eq:25/10-17:11}  is also a solution of~\eqref{eq:init-prob}.

Next, using the change of variable $t=-s$, system~\eqref{eq:25/10-17:11} becomes the time-reversed problem
\[
\dot x= - f\big(x,\overline\Lambda(-rs)\big)
\]
which limits at the autonomous systems $\dot x =-f_+(x)$ as $s\to-\infty$ and to $\dot x =-f(x,\widetilde\lambda)$ as $s\to\infty$, respectively. It is easy to show that the unstable equilibrium $X^u_+$ for the future limit-problem $\dot x =f_+(x)$ of~\eqref{eq:init-prob},  is now a stable equilibrium for the past limit-problem of this time-reversed problem, i.e.~$\dot x =-f_+(x)$ . Therefore, by applying~\cite[Theorem 2.2]{paper:APW} again, one obtains that there is a unique solution $\overline x$ of the time-reversed problem which is locally pullback attracting and limits at $X^u_+$ as $t\to-\infty$. Therefore $\overline x(-t)$ is a locally pullback repelling solution for~\eqref{eq:25/10-17:11} which means that the solution $x^r_+$ of~\eqref{eq:init-prob} such that $x^r_+(t)=\overline x(-t)$ for all $t>(\widetilde\tau+\ep)/r$ is locally pullback repelling for~\eqref{eq:init-prob} (see Definition~\ref{defn:pullback}), which concludes the proof.
\end{proof}

Now, we can rigorously  recall the notion of end-point tracking and of rate-induced tipping \cite{paper:APW,paper:LNOR}. Figure \ref{fig:figure1} shows a graphic representation of the two notions for a scalar quadratic differential problem.
\begin{defn}\label{def:R-tipping}
Under the introduced notation and assumptions and fixed $r>0$:
\begin{itemize}[leftmargin=*]
\item We say that $x^r_-(t)$ \emph{end-point tracks the curve of quasi-static equilibria} $X^s(rt)$ if  $x^r_-(\cdot)$ is defined on the whole real line and 
\[
\lim_{t\to+\infty}x^r_-(t)=X^s_+\,.
\]
\item We say that the system~\eqref{eq:init-prob} \emph{undergoes a rate-induced tipping at $r=r^*\in(0,\infty)$} if for all $r\in(0,r^*)$, all locally pullback attracting solutions constructed as in Proposition~\ref{prop:pbk-attr-rep} end-point track their respective curves of quasi-static equilibria, and there is at least  one locally pullback attracting solution  $x^{r^{\scaleto{*}{3pt}}}_-$ that does not end-point track the respective curve of quasi-static equilibria $X^s$.
\end{itemize}
\end{defn}

\begin{figure}[htbp]
\centering

\begin{overpic}[trim={2.1cm 0cm 2cm 0},clip,width=\textwidth]{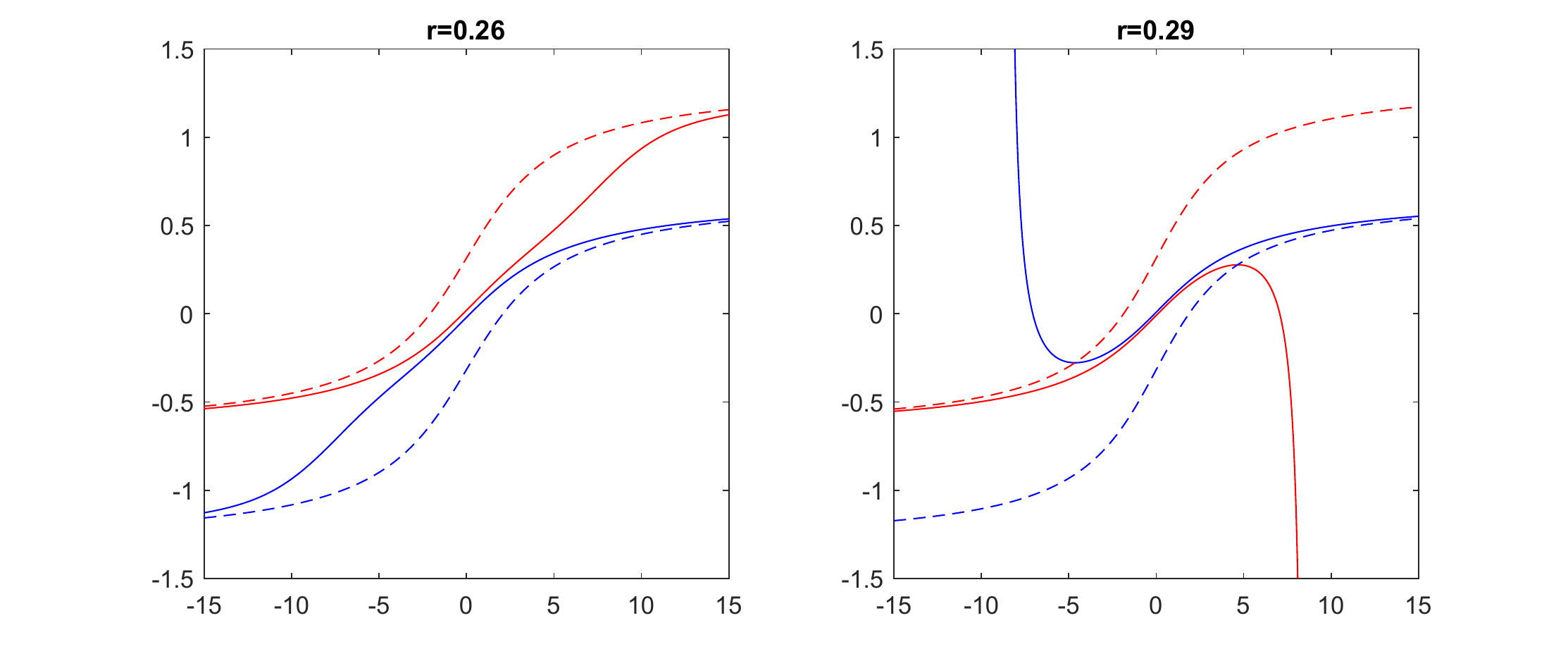}   
\put(35,0){$t$}	
\put(89,0){$t$}
\put(-2,40){$x$}
\put(51,40){$x$}
\end{overpic}
				
\caption{A numerical simulation showing the end-point tracking and  rate-induced tipping for the scalar differential problem $\dot x= -(x-(2/\pi)\arctan(rt))^2+0.1$ upon the variation of the parameter $r$. The solid lines represent the locally pullback attracting and locally pullback repelling solutions of the ODE, namely $x^r_-(t)$ (in red) and  $x^r_+(t)$ (in blue), respectively. The dashed lines represent the curves of quasi-static equilibria $X^{s}(rt)$ (in red) and $X^{u}(rt)$ (in blue) of the associated family of autonomous problems. } \label{fig:figure1}
\end{figure}
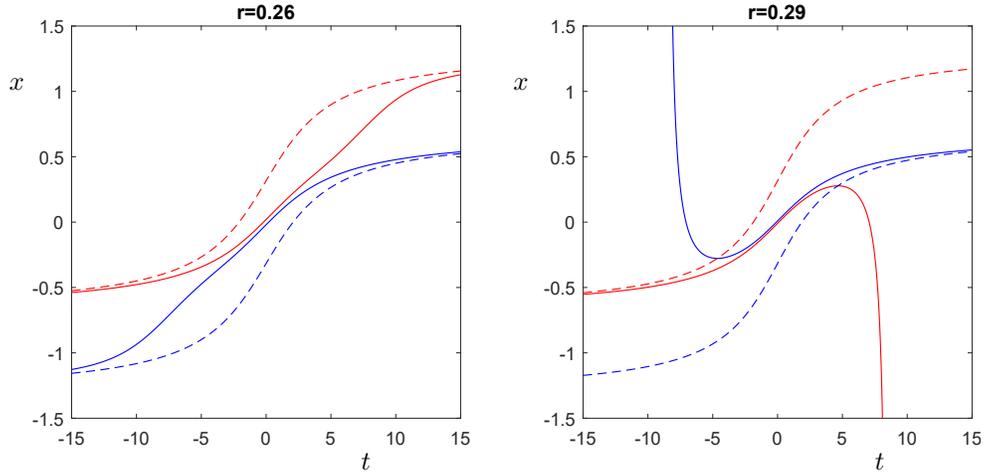

Next, we present a sufficient condition for the end-point tracking of the locally pullback attracting solution $x^r_-$ associated to the curve of quasi-static equilibria $X^s$. Although, the result can be easily generalized to differential problems in $\R^N$, for coherence with this work set-up, we present it in the scalar case. 

\begin{prop}
\label{prop:suff-end-point}
Consider assumption~\ref{H0} and assume that $x^r_-$ is globally defined and bounded for some $r>0$. Let $\cK\subset \R$ be any compact interval such that $X^s_+\in \cK$ and $x^r_-(t)\in \cK$ for all $t\in\R$. Then, a  $T_{r, \cK}>0$ exists such that if for every $\ep>0$ there is $t_\ep>T_{r,\cK}$ for which $|x^r_-(t_\ep)-X^s_+|<\ep$ then $x^r_-$ end point tracks the curve of quasi-static equilibria $X^s$, i.e.~$x^r_-(t)\to X^s_+$ as $t\to\infty$.
\end{prop}

\begin{proof}
The first step is to use a fundamental result of persistence for hyperbolic solutions of non-autonomous dynamical systems. Consider $\alpha>0$ and $h_0\in \mathcal{C}(\R\times\R,\R)$, $(y,t)\mapsto h_0(y,t)$ differentiable in $y$ for every $t\in\R$ such that the differential problem 
\begin{equation}\label{eq:28/10-10:40}
\dot y=-\alpha y+h_0(y,t),
\end{equation}
has an attractive hyperbolic solution $\widetilde y_0:\R\to\R$ in the sense that the corresponding variational equation has an exponential dichotomy on $\R$, i.e.~ there are $k\ge 1 $ and $\beta>0$ such that 
\[
\exp\int_s^t\big[-\alpha +\partial_y h_0\big(\widetilde y(u),u\big)\big]\,  du\le k\, e^{-\beta(t-s)},\quad\text{whenever }t\ge s.
\]
Then, thanks to Proposition 2.1 in~\cite{paper:CLRS} (see also~\cite{paper:CL} or \cite{potz} for a even more general formulation) there is $0<\eta_0\le 1$ such that for every $0<\eta\le\eta_0$ and $\cK'\subset\R$ compact, there is $\delta_{\eta,\cK'}>0$ such that, if $h\in \mathcal{C}(\R\times\R,\R)$, $(y,t)\mapsto h(y,t)$ differentiable in $y$ for every $t\in\R$ satisfies 
\begin{equation}\label{eq:15/10-15:51}
\sup_{y\in \cK'\!,\,t\in\R }|h(y,t)-h_0(y,t)|+ |\partial_y h(y,t)-\partial_yh_0(y,t)|<\delta_{\eta,\cK'},
\end{equation}
then also the equation 
\begin{equation}\label{eq:15/10-16:00}
\dot y=-\alpha y+h(y,t)
\end{equation}
has an attractive hyperbolic solution $\widetilde y:\R\to\R$ and $\|\widetilde y(\cdot)-\widetilde y_0(\cdot)\|_{\mathcal{C}(\R,\R)}<\eta$. In particular, there is a common dichotomy constant pair $(k,\beta)$ valid  for all the hyperbolic solutions obtained for~\eqref{eq:15/10-16:00} whenever $h$ satisfies~\eqref{eq:15/10-15:51}. Moreover, the First Approximation Theorem~\cite[Theorem III.2.4]{book:HALE} guarantees that there is $\mu=\mu(\eta_0)>0$, such that if $|y_0-\widetilde y(\tau)|<\mu$ for some $\tau\in\R$, then $|y(t,h,(\tau,y_0))-\widetilde y(t)|\to0$ as $t\to\infty$, where $y(\cdot,h,(\tau,y_0))$ denotes the solution of~\eqref{eq:15/10-16:00} satisfying $y(\tau,h,(\tau,y_0))=y_0$. \par\smallskip

Now fix $r>0$ and let $\cK_0\subset\R$ be a compact interval containing $X^s_+$ and $x^r_-(t)$ for all $t\in\R$. Notice that if $\cK_0=\{X^s_+\}$ the statement is obvious; therefore, we will assume that $\cK_0$ is strictly larger than $\{X^s_+\}$. Additionally, let  $\cK$ be any compact interval strictly containing $\cK_0$, and consider a function $f_0:\R\times[\lambda_-,\lambda_+]\to\R$, continuously differentiable in $x$ for every $\lambda\in[\lambda_-,\lambda_+]$ such that $f_0(x,\lambda)=f(x,\lambda) $ if $x\in \cK_0$ and  $f_0(x,\lambda)=0 $ if $\R\setminus \cK$. Notice that such a function can always be constructed for example by considering a function which is equal to $f$ in an open neighborhood of $\cK_0$, decays linearly to $0$ towards the border of $\cK$ and realizing its convolution with a suitable mollifier~\cite{book:AFP}. Notably, $x^r_-$ is still a solution of $\dot x = f_0\big(x,\Lambda(rt)\big)$ and  $X^s_+$ is a stable hyperbolic equilibrium for $\dot x =f_0(x,\lambda_+)$ that can be equivalently written as the equation of perturbed motion \eqref{eq:28/10-10:40} with  $-\alpha=\partial_xf(X^s_+,\lambda_+)$, $y=x-X^s_+$ and $h_0(y,t)= f_0(y+X^s_+,\lambda_+)+\alpha y$.  At this point, notice that for every $0<\eta\le \overline \eta=\min\{\eta_0,\mu/3\}$, there is $T_{\eta,r,\cK}>0$ such that for all $ t\ge T_{\eta,r,\cK}$,
\[
\sup_{x \in \R}\big|f_0\big(x,\Lambda(rt)\big)-f_0(x,\lambda_+)\big|+ \big|\partial_xf_0\big(x,\Lambda(rt)\big)-\partial_xf_0(x,\lambda_+)\big|<\delta_{\eta,\cK}.
\]
Therefore, denoted by $h_\eta(y,t)$ the function from $\R\times\R$ into $\R$ defined by 
\[
h_\eta(y,t)=\begin{cases}
f_0\big(y+X^s_+,\Lambda(rT_{\eta,r,\cK})\big)+\alpha y&\text{if }t<T_{\eta,r,\cK},\\
f_0\big(y+X^s_+,\Lambda(rt)\big)+\alpha y &\text{if } t\ge T_{\eta,r,\cK},
\end{cases}
\]
we have that $h_\eta$ satisfies~\eqref{eq:15/10-15:51}, and thus~\eqref{eq:15/10-16:00} (with $h=h_\eta$) has an attractive hyperbolic solution $\widetilde y_\eta:\R\to\R$ such that $\|\widetilde y_\eta(\cdot)\|_{\mathcal{C}(\R,\R)}<\eta$. Notice also that $\widetilde y_\eta$ is a solution of $\dot y= f(y+X^s_+,\Lambda(rt))$ for $t>T_{\eta,r,\cK}$. Furthermore, if $0<\eta_1<\eta_2\le \overline \eta$,  then
\begin{equation}\label{eq:15/10-17:35}
T_{\eta_2,r,\cK}\le T_{\eta_1,r,\cK},\qquad\text{and}\quad\lim_{t\to\infty}|\widetilde y_{\eta_1}(t)-\widetilde y_{\eta_2}(t)|=0,
\end{equation}
because $\widetilde  y_{\eta_2}$ is also a solution of~\eqref{eq:15/10-16:00} with $h=h_{\eta_1}$ for $t>T_{\eta_1,r,\cK}$ and thanks also to the hyperbolicity of $\widetilde y_{\eta_1}$ and the fact that $|\widetilde  y_{\eta_2}(T_{\eta_1,r,\cK})-\widetilde  y_{\eta_1}(T_{\eta_1,r,\cK})|\le \eta_1+\eta_2<\mu$.
Hence, for all $0<\eta\le\overline \eta$, $\widetilde y_\eta(t)\to 0$ as $t\to\infty$. Let $T_{r,\cK}=T_{\overline \eta,r,\cK}$ and assume that for every  $\ep>0$ there is $t_\ep>T_{r,\cK}$ for which $|y^r_-(t_\ep)|=|x^r_-(t_\ep)-X^s_+|<\ep$. Therefore, we have that for any  $0<\eta\le\overline \eta$,
$|y^r_-(t_\eta)- \widetilde y_{\smash{\overline\eta}}(t_\eta)|<\mu.
$
Consequently, 
\[
\lim_{t\to\infty} y^r_-(t)=\lim_{t\to\infty} \widetilde y_{\smash{\overline\eta}}(t)=0,
\]
which concludes the proof.
\end{proof}

Lemma 2.3 in~\cite{paper:APW} proves that there is $\overline r>0$ such that the locally pullback attracting solution $x^r_-$ proceeding from the stable hyperbolic equilibrium $X^s_-$ of the past limit-problem is globally defined and end-point tracks the curve of quasi-static equilibria $X^s(rt)$ for all $0<r<\overline r$.
As follows, we recall the basic ideas for the global existence of $x^r_-$ (see~\cite{paper:APW} for more details) which make use of the theory of persistence for normally hyperbolic invariant non-compact manifolds (see~\cite{book:E}). We also include an alternative argument to prove the end-point tracking condition motivated by the fact that it will be useful in the rest of the paper. Most importantly, we show that for all $0<r<\overline r$, $x^r_-$ is uniformly asymptotically stable.  

\begin{prop}\label{prop:hyperbolic}
Under assumption~\ref{H0}, there is $\overline r>0$ such that if $0<r<\overline r$ then the solutions $x^r_-$ provided by {\rm Proposition~\ref{prop:pbk-attr-rep}} is globally defined, end-point tracks the curve of quasi-static equilibria $X^s(rt)$
and it is uniformly asymptotically stable. 
\end{prop}
\begin{proof}
We only sketch the proof of the global existence of $x^r_-$ (as given in~\cite[Lemma 2.3]{paper:APW})  in order to set some notation. Set $\tau=r t$  and increase the dimension of~\eqref{eq:init-prob}  by adding the equation $\txtd \tau/\txtd t=r$, that is,
\begin{equation}\label{eq:planar_system}
\begin{cases}
\frac{\txtd x}{\txtd t} =f\big(x,\Lambda(\tau)\big),\\
\frac{\txtd \tau}{\txtd t}=r.\\
\end{cases}
\end{equation}
When $r=0$, the flow induced by~\eqref{eq:planar_system} on $\R^2$ is the union over $\tau \in\R$ of the flows on $\R$ induced by the autonomous scalar problems $\frac{\txtd x}{\txtd t}=f(x,\Lambda(\tau))$ in each fiber.
Then, it is possible to prove that $(X^{s}(t),t)$ (which is a smooth connected and complete submanifold of $\R^{N+1}$) is a non-compact \emph{normally hyperbolic invariant manifold} (see~\cite[Definition 1.8]{book:E}) for the flow induced by~\eqref{eq:planar_system} at $r=0$. Furthermore, $(X^{s}(t),t)$ trivially has \emph{empty unstable bundle}. Then, Theorem~3.1  in~\cite{book:E} (with additional details in Section 1.6.1 for the non-autonomous perturbation) guarantees that for each $\ep>0$ there is a $r_\ep>0$ such that for every $0<r<r_\ep$ there is a unique submanifold $(\widetilde x^r(t),rt)$ in the $\ep$-neighborhood of $(X^{s}(rt),rt)$ that is diffeomorphic to $(X^{s}(rt),rt)$ and invariant under the flow induced by~\eqref{eq:planar_system} for the given value of $r$ (see also Section 4.1 in~\cite{book:E}). From Proposition~\ref{prop:pbk-attr-rep} we immediately have that it must be $\widetilde x^r=x^r_-$, and therefore $x^r_-$ is globally defined and bounded. 
\par\smallskip

We shall now prove that for $\ep>0$ sufficiently small,  $x^r_-$ end point tracks $X^s$  for all  $0<r<r_{\ep}$. Fix $\ep>0$, from the first part of the proof we have that 
\[
X^s(rt)-\ep<x^r_-(t)<X^s(rt)+\ep,\quad\text{for all } 0<r<r_\ep, t\in\R.
\]
Denoted by $\underline X:=\liminf_{t\to \infty} x^r_-(t)$ and  $\overline X:=\limsup_{t\to \infty} x^r_-(t)$, there are sequences $(t_n)_\nin,(t_m)_{m\in\N}$ of real numbers so that $t_n\to\infty,t_m\to\infty$ and for all $0<r<r_\ep$,
\[
X^s_+-\ep\le \underline X:=\lim_{m\to \infty} x^r_-(t_m)\le \overline X:=\lim_{n\to \infty} x^r_-(t_n)\le X^s_++\ep.
\]
Since $\ep $ can be taken as small as desired and $X^s_+$ is hyperbolic stable, there is $s>0$ such that 
\[
\max\big\{\big| x\big(s,f_+,\overline X\big)-X^s_+\big|,\big| x \big(s,f_+,\underline X\big)-X^s_+\big|\big\}<\frac{\ep}{2}.
\]
On the other hand, recall that $f_{t_n},f_{t_m}\to f_+$ as $n\to\infty$. Therefore, from Lemma~\ref{lem:kamke}, for the fixed  $s>0$ we have that there is $n_\ep\in\N$ such that for all $n,m\in\N$ with $n,m>n_\ep$,
\[
\big|x\big(s,f_{t_n},x^r_-(t_n)\big)-x\big(s,f_+,\overline X\big)\big|<\frac{\ep}{2},\ \text{ and }\  \big|x\big(s,f_{t_m},x^r_-(t_m)\big)-x\big(s,f_+,\underline X\big)\big|\big\}<\frac{\ep}{2}.
\]
Then, using the cocycle property and the triangular inequality, one has that
\begin{equation}\label{eq:15/09-17:00}
|x^r_-(t_n+s)-X^s_+|\le \big|x\big(s,f_{t_n},x^r_-(t_n)\big)-x\big(s,f_+,\overline X\big)\big|+ \big|x\big(s,f_+,\overline X\big)-X^s_+\big|<\ep
\end{equation}
and analogously, $|x^r_-(t_m+s)-X^s_+|<\ep$. From the arbitrariness of $\ep>0$ we deduce that $\lim_{t\to\infty}x^r_-(t)=X^s_+$. In other words, $x^r_-$ end-point tracks the curve of quasi-static equilibria $X^s(rt)$.
\par\smallskip

In order to prove the uniform asymptotic stability, let us recall that by assumption $\partial_xf\big(X^s(rt),\Lambda(rt)\big)<-p<0$ for all $t\in\R$. Hence, from the first part of this proof and the continuity of $\partial_xf$, we can consider 
\[
0<\overline r := \sup\left\{r>0\ \big|\
\partial_xf\big( x^r_-(t),\Lambda(rt)\big)<0\,\text{ for all }t\in\R\right\}.
\]
For any fixed $0<r<\overline r$ and  $\ep>0$, consider the two continuous functions $\alpha^r_\ep,\beta^r_\ep\colon\R\to\R$ defined by
\[
t\mapsto\alpha^r_\ep(t)=x^r_-(t)-\ep,\quad\text{and}\quad t\mapsto\beta^r_\ep(t)=x^r_-(t)+\ep.
\]
Notice that 
\[\begin{split}
f\big(\beta^r_\ep(t),\Lambda(rt)\big)< \dot\beta^r_\ep(t)\ &\Leftrightarrow\  f\big(x^r_-(t)+\ep,\Lambda(rt)\big)<\dot x^r_-(t)\\
&\Leftrightarrow\ \ep \partial_xf\big(x^r_-(t),\Lambda(rt)\big)+ O(\ep^2)<0,
\end{split}
\]
and analogously,
\[
\ \ f\big(\alpha^r_\ep(t),\Lambda(rt)\big)> \dot \alpha^r_\ep(t)  \ \Leftrightarrow\  -\ep \partial_xf\big(x^r_-(t),\Lambda(rt)\big)+ O(\ep^2)> 0.
\]
Both the inequalities at the right-hand side of the previous chain of implications are always true if $\ep>0$ is sufficiently small. Therefore, $x^r_-(t)$ is uniformly asymptotically stable if $0<r<\overline r$. 
\end{proof}
Next we study a bit more in depth the region of the phase space which asymptotically converges to $x^r_-$ if it is uniformly asymptotically stable. In particular, we show that, under assumption~\ref{H0}, if there is an unstable hyperbolic  equilibrium $X^u_+$ for the future limit-problem which belongs to the boundary of the basin of attraction of the considered stable hyperbolic equilibrium $X^s_+$ and the locally pullback attracting solution $x^r_-$ converges to $X^s_+$, then the locally pullback repelling solution associated to $X^u_+$ (provided by Proposition~\ref{prop:pbk-attr-rep}) determines a region of the phase space which converges to $x^r_-$ as $t\to\infty$.
\begin{prop}
\label{prop:go-up}
Under assumption~\ref{H0}, let $x^r_-$ be the locally pullback attracting solution provided by {\rm Proposition~\ref{prop:pbk-attr-rep}} and assume that $x^r_-$ end-point tracks the associated curve of quasi-static equilibria and it is uniformly asymptotically stable for some $r>0$. Assume $\dot x =f(x,\lambda_+)$ has also an unstable hyperbolic equilibrium $X^u_+<X^s_+$ (resp.~$X^s_+<X^u_+$) such that there are no other equilibria between $X^u_+$ and $X^s_+$. Considering the pullback repelling solution $x^r_+(\cdot)\in \mathcal{C}((\beta^r_+,\infty),\R^N)$ of~\eqref{eq:init-prob} provided by {\rm Proposition~\ref{prop:pbk-attr-rep}}, we have that for every $t_0>\beta^r_+$, if $x^r_+(t_0)<x_0<x^r_-(t_0)$ (resp.~$x^r_-(t_0)<x_0<x^r_+(t_0)$), then
\begin{equation*}
 |x(t,t_0,x_0)-x^r_-(t)|\to0 \quad\text{ as $t\to\infty$.}
\end{equation*}
\end{prop}

\begin{proof}
We shall complete the proof for the case $X^u_+<X^s_+$, the other case being similar.  Let us fix $\ep>0$ and consider $t_\ep>0$ such that
\[
|X^s_+- x^r_-(t)|<\frac{\ep}2,\quad\text{for all }t\ge t_\ep.
\]
Notice also that, chosen $(x_0,t_0)\in\R^2$  as in the assumptions, the solution $x(t,t_0,x_0)$ of~\eqref{eq:init-prob} is defined for all $t>t_0$ since
\[
x^r_+(t)<x(t,t_0,x_0)<x^r_-(t),\quad\text{for all }t>t_0.
\]
Denote by $\underline x=\liminf_{t\to\infty} x(t,t_0,x_0)$ and consider a sequence $(t_n)_\nti$, realizing $\underline x=\lim_\nti x(t_n,t_0,x_0)$. Due to Proposition~\ref{prop:pbk-attr-rep} we have that $X^u_+<\underline x\le X^s_+$. Therefore, for the fixed $\ep>0$, we have that there is $s>t_\ep$ such that 
\[
\big| x \big(s,f_+,\underline x\big)-X^s_+\big|<\frac{\ep}{4}.
\]
On the other hand, recall that $f_{t_n}\to f_+$ as $n\to\infty$. Therefore, from Lemma~\ref{lem:kamke}, for the fixed  $s>t_\ep$ we have that there is $n_\ep\in\N$ such that for all $n\in\N$ with $n>n_\ep$,
\[
  \big|x\big(s,f_{t_n},x(t_n,t_0,x_0)\big)-x\big(s,f_+,\underline x\big)\big|<\frac{\ep}{4}.
\]
Then, using the cocycle property and the triangular inequality, one has that
\[
\begin{split}
|x(t_n+s,t_0,x_0)-x^r_-(t_n+s)|\le
|x(t_n+s,t_0,x_0)-X^s_+|+|X^s_+-x^r_-(t_n+s)|&\\
\le \big|x\big(s,f_{t_n},x(t_n,t_0,x_0)\big)-x\big(s,f_+,\underline x\big)\big|
+ \big|x\big(s,f_+,\underline x \big)-X^s_+\big|+\frac{\ep}{2}<\ep.&    
\end{split}
\]
Finally, since $x^r_-$ is assumed to be uniformly asymptotically stable and from the arbitrariness on $\ep>0$, we obtain the desired conclusion.
\end{proof}

The following result characterizes the occurrence of a rate-induced tipping with the loss of uniform asymptotic stability of a locally pullback attracting solution limiting at a stable hyperbolic equilibrium of the past limit-problem (see Proposition~\ref{prop:hyperbolic}). Many works have shown that a rate-induced tipping coincides with a collision with an orbit which is defined up to $+\infty$ and limits at a repelling solution of the future limit-problem~\cite{paper:WXJ, paper:Ca, paper:LNOR}. This result contains also this fact  in the context of scalar differential problems using standard techniques for non-autonomous dynamical systems.  

\begin{thm}\label{thm:fundamentals}
Consider assumption~\ref{H0} and denote 
\[
r^*:=\sup\left\{r>0\mid
  x^\rho_-(\cdot) \text{ is uniformly asymptotically stable  for all }0<\rho\le r
\right\}\!.
\]
Then, $x^r_-$ end point tracks the respective curve of quasi-static equilibria $X^s$ for all  $0<r<r^*$. Moreover, the following statements are equivalent.
\begin{itemize}[leftmargin=20pt,itemsep=2pt]
\item[\rm (i)] $r^*<\infty$;
\item[\rm (ii)] $x^{r^{\scaleto{*}{3pt}}}_-$ is globally defined but not uniformly asymptotically stable and it does not end-point track $X^s$. In particular, the system~\eqref{eq:init-prob} undergoes a rate-induced tipping at $r=r^*$;
\item[\rm (iii)] there is an unstable hyperbolic equilibrium $X^u_+$ of $\dot x =f(x,\lambda_+)$ and a value $r^*>0$ such that the  pullback repelling solution $x^{r^{\scaleto{*}{3pt}}}_+$ of {\rm~\eqref{eq:init-prob}} associated to  $X^u_+$ (see {\rm Proposition~\ref{prop:pbk-attr-rep}}) satisfies $x^{r^{\scaleto{*}{3pt}}}_-(t)=x^{r^{\scaleto{*}{3pt}}}_+(t)$ for all $t\in\R$.
\end{itemize}
\end{thm}

\begin{proof}
From Proposition~\ref{prop:hyperbolic}, we already know that $\lim_{t\to\infty}x^r_-(t)=X^s_+$ for all $0<r<\overline r$. If $r^*=\overline r$ we already have the first assertion proved. Otherwise, consider $r\in[\overline r,r^*)$ and let
\begin{equation}\label{eq:16/09-17:52}
\underline X\vphantom{X}^r:=\liminf_{t\to \infty} x^r_-(t)\le \smash{\overline X}\vphantom{X}^r:=\limsup_{t\to \infty} x^r_-(t).
\end{equation}
If we prove that both $\underline X\vphantom{X}^r$ and $\smash{\overline X}\vphantom{X}^r$ must belong to the basin of attraction of $X^s_+$ for all $0<r<r^*$, then  reasoning as for~\eqref{eq:15/09-17:00}, we immediately obtain that $\underline X\vphantom{X}^r=\smash{\overline X}\vphantom{X}^r=X^s_+$ which is the claimed assertion. By contradiction, assume that this is not true and there is $r_0\in[\overline r,r^*)$ and a sequence $(t_n)_\nin$, $t_n\xrightarrow{\nti}\infty$ such that $\underline X\vphantom{X}^{r_0}=\lim_{\nti}x^{r_0}_-(t_n)$ does not belong to the basin of attraction of $X^s_+$. Then, due to~\ref{H0}, there must be an unstable hyperbolic equilibrium $X^u_+$ in between $\underline X\vphantom{X}^{r_0}$ and $X^s_+$ for the future limit-problem $\dot x=f_+(x)$. For the sake of notation, let us assume  that $\underline X\vphantom{X}^{r_0}\le X^u_+<X^s_+$ (the other case being similar), and denote by $x^{r_0}_+$ the locally pullback repelling solution of~\eqref{eq:init-prob} provided by Proposition~\ref{prop:pbk-attr-rep}.  Then, one has that, up to considering $n\in\N$ sufficiently big, $x^r_+(t_n)<x^r_-(t_n)$ for all $0<r<\overline r$,  and  $x^{r_0}_+(t_n)\ge x^{r_0}_-(t_n)$. Due to the continuous variation of the solution there must be $r_1\in [\overline r, r_0]$ such that  $x^{r_1}_+(t_n)=x^{r_1}_-(t_n)$. But this is in contradiction with the fact that $x^{r}_-$ is uniformly asymptotically stable for all $0<r<r^*$ because, by Proposition~\ref{prop:pbk-attr-rep}, $x^{r}_+$ is locally pullback repelling for all $r>0$. An analogous reasoning holds true for $\smash{\overline X}\vphantom{X}^r$. Hence, both $\underline X\vphantom{X}^r$ and $\smash{\overline X}\vphantom{X}^r$ belong to the basin of attraction of $X^s_+$ for all $0<r<r^*$, and, in particular, coincide with $X^s_+$ (see~\eqref{eq:15/09-17:00}), which ends the proof of the first assertion.\par\smallskip

 (i) $\Rightarrow$ (ii).  First of all, notice that due to Lemma~\ref{lem:kamke}, $x^{r^{\scaleto{*}{3pt}}}_-$ must be globally defined.  Moreover,  due to assumption~\ref{H0}, $f\big(x,\Lambda(rt)\big)$ is locally Lipschitz-continuous in $x$ uniformly in $t\in\R$. Therefore, $x^{r^{\scaleto{*}{3pt}}}_-$ can not be uniformly asymptotically stable or else  $x^{r^{\scaleto{*}{3pt}}+\delta}_-$ would be so for $\delta>0$ sufficiently small (see~\cite[Theorem X.5.2]{book:HALE}), and this is in contradiction with the definition of $r^*$. On the other hand, this implies that $x^{r^{\scaleto{*}{3pt}}}_-$ does not end-point track $X^s$, otherwise this would contradict Proposition~\ref{prop:asympt-stability}. 
 \par\smallskip
 
 (ii) $\Rightarrow$ (iii). Using the notation of~\eqref{eq:16/09-17:52}, we have that either $\underline X\vphantom{X}^{r^{\scaleto{*}{3pt}}}$, $\smash{\overline X}\vphantom{X}^{r^{\scaleto{*}{3pt}}}$, or both must not belong to the basin of attraction of $X^s_+$. Otherwise reasoning as for~\eqref{eq:15/09-17:00}, we obtain that $\underline X\vphantom{X}^{r^{\scaleto{*}{3pt}}}\!=\smash{\overline X}\vphantom{X}^{r^{\scaleto{*}{3pt}}}\!=X^s_+$ which is in contradiction with the assumptions in~(ii). Therefore, in agreement with~\ref{H0}, there must be an unstable hyperbolic equilibrium $X^u_+$ for $\dot x= f_+(x)$ such that the only one of the following two cases is possible
 \begin{equation}\label{eq:03/10-17:24}
 \smash{\overline X}\vphantom{X}^{r^{\scaleto{*}{3pt}}}\ge X^u_+>X^s_+,\quad\text{ or }\quad X^s_+> X^u_+\ge\underline X\vphantom{X}^{r^{\scaleto{*}{3pt}}}.
\end{equation}
  Let us assume the latter (the other case being similar) and  denote by $x^{r^{\scaleto{*}{3pt}}}_+$ the locally pullback repelling solution of~\eqref{eq:init-prob} associated to $X^u_+$ at $r=r^*$  (see Proposition~\ref{prop:pbk-attr-rep}). Thanks to Lemma~\ref{lem:kamke}, it is easy to deduce that it must be $x^{r^{\scaleto{*}{3pt}}}_-(t)\ge x^{r^{\scaleto{*}{3pt}}}_+(t)$ for all $t\in\R$ where they are both defined. 

Consequently, and due also to~\eqref{eq:03/10-17:24}, it must be $\underline X\vphantom{X}^{r^{\scaleto{*}{3pt}}}=X^u_+$. We shall now prove that also $ \smash{\overline X}\vphantom{X}^{r^{\scaleto{*}{3pt}}}= X^u_+<$. Let us assume that this were not true. Then, two cases are possible: either $ \smash{\overline X}\vphantom{X}^{r^{\scaleto{*}{3pt}}}$ belongs to the basin of attraction of $X^s_+$ (and therefore reasoning as in ~\eqref{eq:15/09-17:00}, we would have that $\smash{\overline X}\vphantom{X}^{r^{\scaleto{*}{3pt}}}=X^s_+$) or there is a second unstable equilibrium $Y^u_+$  for $\dot x= f_+(x)$ such that $\smash{\overline X}\vphantom{X}^{r^{\scaleto{*}{3pt}}}\ge Y^u_+>X^s_+$. 
Either way, for every $\ep>0$, there is a sequence $(t_n)_\nin$, $t_n\to\infty$ such that $|x^{r^{\scaleto{*}{3pt}}}_-(t_n)-X^s_+|<\ep$. Then, by Proposition~\ref{prop:suff-end-point} one would have that $x^{r^{\scaleto{*}{3pt}}}_-(t)\to X^s_+$ as $t\to\infty$ which is in contradiction with~\eqref{eq:03/10-17:24}. Therefore, it must be $ \smash{\overline X}\vphantom{X}^{r^{\scaleto{*}{3pt}}}= X^u_+$ and thus $x^{r^{\scaleto{*}{3pt}}}_-(t)\to X^u_+$ as $t\to \infty$. However, in this case, due to Proposition~\ref{prop:pbk-attr-rep}, we have that $x^{r^{\scaleto{*}{3pt}}}_-= x^{r^{\scaleto{*}{3pt}}}_+$ as claimed.
\par\smallskip
(iii) $\Rightarrow$ (i) is trivial.
\end{proof}

\section{Asymptotic series expansions and rate-induced tipping}\label{sec:asympt_series_exp}

The role played by the (locally) pullback attractive and repelling solutions $x^r_-$ and $x^r_+$ in the phenomenology of a rate-induced tipping has been stressed in several works~(see~\cite{paper:APW, paper:LNOR,paper:WXJ} and references therein).  In this section, we provide a method of approximation of  $x^r_-$ and $x^r_+$ via asymptotic series expansions (in the sense of Definition~\ref{def:asymptoticseries} below), which, to our knowledge, has not been used in the context of rate-induced tipping. Nevertheless, we show that such approximations, which are always numerically calculable, also provide valuable information on the occurrence or absence of rate-induced tipping for scalar differential problems. On the one hand, we immediately obtain a sufficient condition for end-point tracking in Proposition~\ref{prop:suff-end-point} if the series approximation holds on the whole real line and the relative error is suitably bounded although possibly not small. Theorem~\ref{thm:melnikov}, on the other hand, provides a characterization of the occurrence of a rate-induced tipping as well as a asymptotic approximation of the tipping value if the asymptotic series expansions are reliable only on half-lines (in the sense of Proposition~\ref{prop:expansions}) but the solutions $x^r_-$ and $x^r_+$ are, respectively, locally pullback attractive and repelling at least up to $t=0$.  At the end of the section, we sum up the advantages and limitations offered by our results.

\begin{defn}
\label{def:asymptoticseries}
Consider $\cI \subseteq\R$, $\ep_0>0$ and a function $v:\cI\times [0,\ep_0)\to\R$. If for all $\nin$ there is $\delta_n\colon\R^+\to\R^+$ continuous and vanishing at zero, a constant $C_n>0$ and a function $a_n:\cI\to\R$, such that
\[
\Big\|v(t,\ep)-\sum_{i=0}^na_i(t)\delta_i(\ep)\Big\|_{\mathcal{C}(\cI,\R)}\!\!<C_n\delta_n(\ep),\quad\text{then}\quad v(t,\ep)\sim\sum_{i=0}^{\infty}a_i(t)\delta_i(\ep),
\]
is called an \emph{asymptotic series of} $v$.
\end{defn}

It is important to keep in mind that an asymptotic series is in general not convergent in the classical sense. The term asymptotic suggests that the approximation provided by a finite sum of its elements becomes accurate as $\ep\to0$.\par\smallskip

Besides the assumptions~\ref{H0} considered in Section~\ref{sec:prelim}, in the following we shall assume also that: \par\smallskip

\setword{\upshape(\textbf{H1})}{H1} the  future limit-problem $\dot x=f(x,\lambda_+)$ has an unstable hyperbolic fixed point $X^u_+$ such that, upon the variation on $ \lambda$, $X^u_+$ is (continuously) perturbed into a hyperbolic fixed point $X^u_\lambda$ of $\dot x=f(x,\lambda)$, with $ \lambda\in [\lambda_-,\lambda_+]$. Again, we simplify the notation at $\lambda=\lambda_-$ by writing $X^{u}_-:=X^{u}_{\lambda_-}$. We shall call the function $X^u:[\lambda_-,\lambda_+]\to\R$ defined by 
\[
\lambda\mapsto X^u(\lambda)=X^u_\lambda,
\]
the family of \emph{quasi-static unstable equilibria} associated to $X^u_+$. Moreover, assume that, for all $ \lambda\in [\lambda_-,\lambda_+]$, there are no further fixed points in the open interval of  end-points  $X^u(\lambda)$ and $X^s(\lambda)$. In particular, we will assume that there exists $d_0>0$ such that given any pair of curves of quasi-static equilibria $X,Y:[\lambda_-,\lambda_+]\to\R$,  if $X\neq Y$,  then
\[
\min_{\lambda\in[\lambda_-,\lambda_+]}|X(\lambda)-Y(\lambda)|>d_0.
\]
From the previous assumptions one immediately has that for all $\lambda\in[\lambda_-,\lambda_+]$, either $X^s(\lambda)<X^u(\lambda)$ or $X^s(\lambda)>X^u(\lambda)$. We will assume the latter (the other case being similar). \par\smallskip

 We advice the reader that for the upcoming results, anytime \ref{H1}  is in force, so is \ref{H0}. This justifies the absence of assumptions on the regularity of $f$ in \ref{H1}. We shall see, however, that the results in this section will require that the considered vector field admits partial derivatives of any order with respect to $x$.
\par\smallskip
The following result shows that, under appropriate asssumptions, the locally pullback solutions of~\eqref{eq:init-prob} provided by \rm Proposition~\ref{prop:pbk-attr-rep} can be written as asymptotic series. 

\begin{prop}
\label{prop:expansions}
Consider the problem {\rm~\eqref{eq:init-prob}}, with its relative assumptions~\ref{H0}. Assume also that $f$  admits partial derivatives of any order with respect to $x$ and they are bounded along $\big(X^s(\tau),\Lambda(\tau)\big)$, $\tau\in \R$. The following statements are true.
\begin{itemize}[leftmargin=18pt,itemsep=2pt]
\item[\rm (i)]There exists a sequence $(a_m^{s})_{n\in\N}$ of bounded continuous functions from $\R$ into itself, and sequences  $(C^s_n)_\nin$ and $(\overline r^s_n)_\nin$ of positive real numbers, such that for every fixed $\nin$, if  $0<r<\overline r^s_n$,
\begin{equation}\label{eq:approx_are_close}
\Big\|\sum_{i=0}^n a_i^s(r\,\cdot\,)r^i -x^r_-(\cdot)\Big\|_{\mathcal{C}(\R)}\!\!\!<C^s_nr^{n+1},
\end{equation}
that is, the solution $x^r_-$ of~\eqref{eq:multisc-slow} provided by {\rm Proposition~\ref{prop:pbk-attr-rep}} can be written as asymptotic series. In particular, one has that
\begin{equation*}
x^{r}_-(t)\sim X^s(rt)+\sum_{i=1}^\infty a_i^s(rt)r^i, 
\end{equation*}
where all the coefficients $a_i^{s}(\cdot)$ can be calculated using function values and derivatives of $f$ and $\Lambda$. Moreover, for any fixed $\ep>0$, $\nin$ and $r>0$ there exists $\beta=\beta(\ep,n,r)\in\R\cup\{\infty\}$ so that
\begin{equation}\label{eq:approx-half-lines1}
\Big\|\sum_{i=0}^n a_i^s(r\,\cdot\,)r^i -x^r_-(\cdot)\Big\|_{\mathcal{C}((-\infty,\beta))}<\ep
\end{equation}
In particular for all $\nin$, $\beta(\ep,n,r)=\infty$ whenever $r<\min\big\{\overline r^s_n,\sqrt[n+1]{\ep/C^s_n}\big\}$.

\item[\rm (ii)] If additionally {\rm~\eqref{eq:init-prob}} satisfies also~\ref{H1} and all the partial derivatives of $f$ with respect to $x$ are bounded along  $\big(X^u(\tau),\Lambda(\tau)\big)$, $\tau\in \R$,  then an analogous result is valid for $x^r_+$. That is, there is a sequence $(a_n^{u})_{n\in\N}$ of bounded continuous functions from $\R$ into itself,  and sequences  $(C^u_n)_\nin$ and $(\overline r^u_n)_\nin$ of positive real numbers such that for every fixed $\nin$, if  $0<r<\overline r^u_n$,
\[
\ \quad\Big\|\sum_{i=0}^n a_i^u(r\,\cdot\,)r^i -x^r_+(\cdot)\Big\|_{\mathcal{C}(\R)}\!\!\!<C^u_nr^{n+1}
\quad\text{and}\quad 
x^{r}_+(t)\sim X^u(rt)+\sum_{i=0}^\infty a_i^u(rt)r^i,
\]
where all the coefficients $a_i^{u}(\cdot)$ can be calculated using function values and derivatives of $f$ and $\Lambda$. Moreover, for any fixed $\ep>0$, $\nin$ and $r>0$ there exists $\alpha=\alpha(\ep,n,r)\in\R\cup\{-\infty\}$  such that
\begin{equation}\label{eq:approx-half-lines}
\Big\|\sum_{i=0}^n a_i^u(r\,\cdot)r^i -x^r_+(\cdot)\Big\|_{\mathcal{C}((\alpha,\infty))}<\ep.
\end{equation}
In particular, for all $\nin$, $\alpha(\ep,n,r)=-\infty$  whenever  $r< \min\big\{\overline r^u_n,\sqrt[n+1]{\ep/C^u_n}\big\}$.
\end{itemize}
\end{prop}

\begin{proof}
The proof uses several ideas presented in~\cite[Ch.~5]{book:K}. We shall prove (i) and omit the proof of (ii) because it is analogous. Before proceeding, let us highlight the multi-scale structure of the considered differential problem. Set $\tau=r t$  and increase the dimension of~\eqref{eq:init-prob}  by adding the equation $\txtd \tau/\txtd t=r$ as in~\eqref{eq:planar_system}. Then, on the slow time-scale, obtained via the change of variable $s=rt$, one has that~\eqref{eq:init-prob} can be rewritten as 
\begin{equation}\label{eq:multisc-slow}
\begin{array}{rcl}
r\frac{\txtd x}{\txtd s} &=&f\big(x,\Lambda(\tau)\big),\\
\frac{\txtd \tau}{\txtd s}&=&1.\\
\end{array}
\end{equation}
Now, consider a formal series solution of~\eqref{eq:multisc-slow}, of the form,
\[
\sum_{i=0}^\infty a_i(\tau)r^i\,,
\]
and plug it in~\eqref{eq:multisc-slow}. On the left-hand side, differentiate term-by-term, while on the right-hand side use the multidimensional Taylor formula 
\begin{equation}\label{eq:TaylorFormula}
F\Big(z_0+\sum_{i=1}^\infty z_ir^i\Big)=F(z_0)+\sum_{i=1}^\infty r^i\sum_{j=1}^i\frac{F^{(j)}(z_0)}{j!}\sum_{k_1+\dots+k_j=i,\,k_l\ge1}z_{k_1}\dots z_{k_j}\,.
\end{equation}
Assuming that everything is well-defined and gathering together the terms on each side which are multiplied by the same power of $r$, one obtains a family of algebraic equations.
Specifically, comparing the coefficient of the powers of $r^0=1$ on both sides we obtain
\[
0=f\big(a_0(\tau),\Lambda(\tau)\big)\,,
\]
which trivially holds true when $a_0(\cdot)=X^s(\cdot)$  Note that if~\ref{H1} is satisfied then also $a_0(\cdot)=X^u(\cdot)$ would satisfy the previous equality). We proceed by considering $a_0(\cdot)=X^s(\cdot)$ not as a variable but as a known term and therefore we rename it as $a_0^s(\cdot)$. 
Then, regarding the coefficient of the term $r$ one has to solve the algebraic equation
\[
\dot a^s_0(\tau)= \partial_xf\big(a^s_0(\tau),\Lambda(\tau)\big)\,a_1(\tau),\quad\text{where } a^s_0(\cdot)=X^s(\cdot).
\] 
Notice that, since for any fixed $\tau\in\R$, $X^s(\tau)$ is an hyperbolic fixed point, then $\partial_xf\big(X^{s}(\tau),\Lambda(\tau)\big)\neq0$ for all $\tau \in\R$.  We shall call $a_1^s(\cdot)$ the obtained coefficient, recalling that it depends on the fact that $a_0^s(\cdot)=X^s(\cdot)$. This reasoning can be iterated so that one obtains a sequence of coefficients $(a^s_i)_{i\in\N}$ such that  for all $i\ge2$ one has 
\begin{equation*}\label{eq:21/05-12:21}
a_i^s(\tau)=\partial_xf^{-1}\Bigg[\dot a_{i-1}^s(\tau)- \sum_{j=2}^i\frac{1}{j!}\frac{\partial^j f}{\partial x^j}\sum_{\substack{k_1+\dots+k_j=i,\\ k_l\ge1}}a_{k_1}^s(\tau)\dots a_{k_j}^s(\tau)\Bigg],
\end{equation*}
where all the derivatives of $f$ are evaluated at $(a^s_0(\tau),\Lambda(\tau))$, and $ a_0^s(\cdot)=X^s(\cdot)$.  The notation $a^s_i$ is used to remind the reader that the obtained coefficients depend upon the initial choice 
 $a_0^s(\cdot)=X^s(\cdot)$ (for example, if \ref{H1} holds and one sets $a_0(\cdot)=X^u(\cdot)$ then a different sequence of coefficients $(a^u_i)_{i\in\N}$ can be constructed using the same method).
Reasoning by induction, it is easy to prove that for every $\nin$ there is $M_n>0$ such that $\|a^s_{n}\|_{\mathcal{C}(\R,\R)}<M_n$ and also that if $n\ge 1$, $\lim_{\tau\to\pm\infty}\dot a_{n-1}(\tau)=\lim_{\tau\to\pm\infty} a_{n}(\tau)=0$.\par\smallskip

We now prove that the above-constructed coefficients $(a^s_{n})_\nin$ allow us to obtain an asymptotic series expansion of $x^r_-$. In order to compare the so-constructed series with $x^r_-$, a new change of variable to the fast time-scale is necessary. There is no problem in doing so, since the coefficients $a^s_i$ are defined on the whole real line. Moreover, in order to simplify the notation, the symbol $S_n^s(r,t)$ will represent the partial sum
\[
S_n^s(r,t)=\sum_{i=0}^na_i^s(rt)r^i.
\]
Let us fix any $\nin$. Since $\partial_xf\big(X^s\big(\Lambda(rt)\big),\Lambda(rt)\big)<0$  for all $t\in\R$,  all the coefficients $(a^s_{n})_\nin$ are continuous and bounded, and since $\partial_xf$ is continuous, there exists $\overline r^s_n>0$ such that, 
\begin{equation}\label{eq:r_n}
\overline r^s_n=\sup\left\{r>0\ \Bigg|\
\partial_xf\big( S_n^s(r,t),\Lambda(rt)\big)<0\,\text{ for all }t\in\R\right\}.
\end{equation}
For any fixed $0<r<\overline r^s_n$ and  $C>0$, consider the two continuous functions $\gamma_1^r,\gamma_2^r\colon\R\to\R$ defined by
\[
t\mapsto\gamma_1^r(t)=S_n^s(r,t)-Cr^{n+1},\quad\text{and}\quad t\mapsto\gamma_2^r(t)=S_n^s(r,t)+Cr^{n+1}.
\]
Notice that, since  by construction
\begin{equation}\label{eq:31/03-19:40}
\lim_{t\to\pm\infty}S_n^s(r,t)=\lim_{t\to\pm\infty} X^s(rt)=X^s_\pm,
\end{equation}
then  there is $T(n,r,C)\in\R$ such that $x^r_-$ is defined at least until $t=T(n,r,C)$ and  $\gamma_1^r(t)\le x^r_-(t)\le\gamma_2^r(t)$ for all $t<T(n,r,C)$. In fact, we aim at finding a constant $C^s_n>0$ so that for all $s\in\R$ if  $\gamma_1^r(s)\le x_0\le\gamma_2^r(s)$ then $\gamma_1^r(t)\le x(t,s,x_0)\le\gamma_2^r(t)$ for all $t\ge s$ and $0<r<\overline r^s_n$ which would end the proof of the first statement.

To the aim, notice that $f(\gamma_1^r(t),\Lambda(rt))>\dot \gamma_1^r(t)$ if and only if
\[
f\big(S_n^s,\Lambda\big)-Cr^{n+1}\partial_xf\big(S_n^s,\Lambda\big)+ O(r^{n+2})>\frac{\txtd }{\txtd t}S_n^s= \sum_{i=0}^n \frac{\txtd  a_i^s}{\txtd t}r^{i+1},
\]
where $S_n^s$ is always evaluated at $(r,t)$ and the remaining functions at $rt$. 
On the other hand, from the fact that $\sum_{i=0}^\infty a_i^s(rt)r^i$ is assumed to be a formal solution of ~\eqref{eq:multisc-slow} and using~\eqref{eq:TaylorFormula}, we can write
\[
\sum_{i=0}^n \frac{\txtd  a_i^s}{\txtd t}r^{i+1}+ O(r^{n+2})=f\big(S_n^s,\Lambda\big)+a^s_{n+1}r^{n+1}\partial_xf\big(S_n^s,\Lambda\big)+O(r^{n+2}).
\]
where, once again, $S_n^s$ is evaluated at $(r,t)$ and all the other functions at $rt$.
\begin{figure}[htbp]
\centering
\begin{overpic}[trim={3.1cm 0.7cm 2.7cm 0.2cm},clip,width=\textwidth]{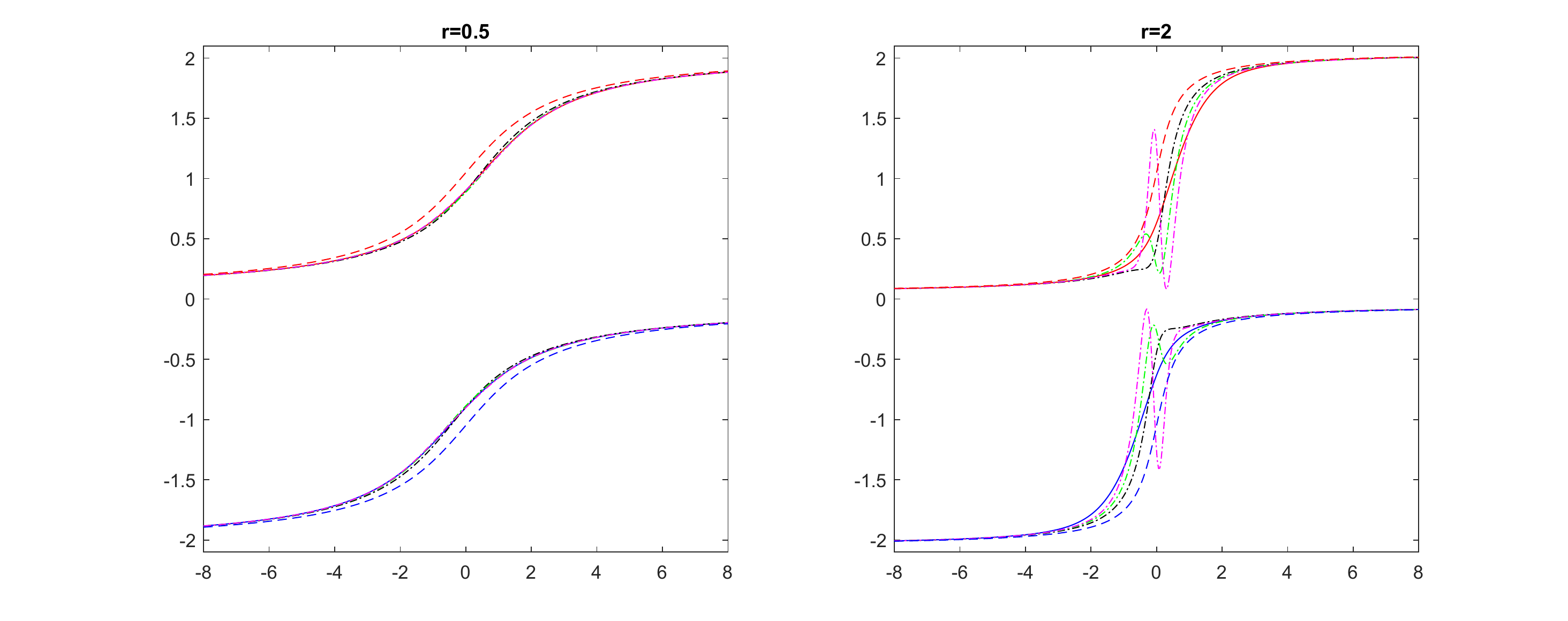}   
\put(36,-1){$t$}	
\put(91,-1){$t$}
\put(-2,40){$x$}
\put(51,40){$x$}
\end{overpic}

\caption{\label{fig:2}A numerical simulation showing the approximations of order up to three for the pullback attracting and repelling solutions of  $\dot x= -(x-(2/\pi)\arctan(rt))^2+1.1$ for $t\in[-8,8]$ and $r=0.5$ on the left and $r=2$ on the right. The solid lines represent the locally pullback attracting solution $x^r_-(t)$ (in red) and locally pullback repelling solution $x^r_+(t)$ (in blue). The dashed lines represent the curves of quasi-static equilibria $X^{s}(rt)$ (in red) and $X^{u}(rt)$ (in blue) of the associated family of autonomous problems. The dot-dashed lines are the approximations $S^s_n(r,t)$ and $S^u_n(r,t)$ calculated using Proposition~\ref{prop:expansions}; $n=1$ in black, $n=2$ in green and $n=3$ in magenta. The picture on the right-hand side shows how the approximations behave always correctly on suitable half-lines but  cease to be reliable over the whole real line as $r$ increases.}
\end{figure}
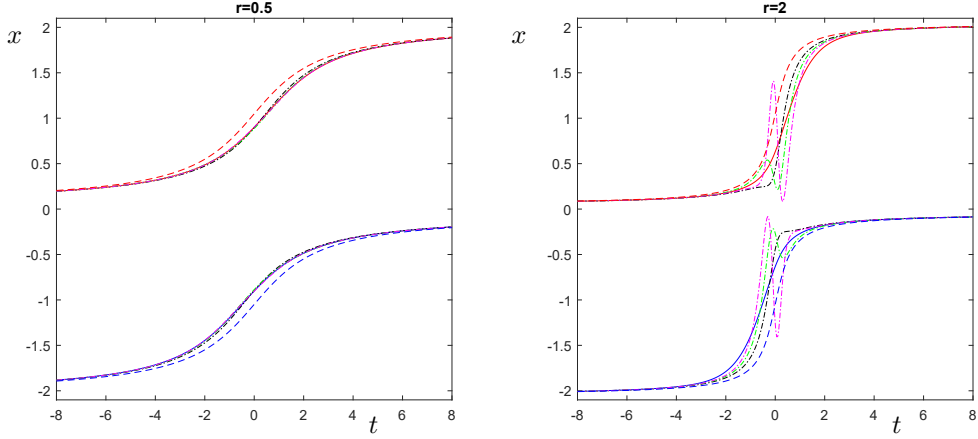 
Then, gathering the previous information one has that $f(\gamma_1^r(t),\Lambda(rt))>\dot \gamma_1^r(t)$ if
\begin{equation}\label{eq:26/10-18:24}
-(C+a^s_{n+1})r^{n+1}\partial_xf\big(S_n^s,\Lambda\big)+O(r^{n+2})>0.
\end{equation}
The first term in this expression determines its sign as $r\to0$. Reasoning analogously for the curve $\gamma_2^r$ we have that $f(\gamma_2^r(t),\Lambda(rt))<\dot \gamma_2^r(t)$ if
\[
(C+a^s_{n+1})r^{n+1}\partial_xf\big(S_n^s,\Lambda\big)+O(r^{n+2})<0,
\]
From~\eqref{eq:r_n} and $\|a^s_{n}\|_{\mathcal{C}(\R)}<M_n$, it is sufficient to consider $C^s_n$ big enough so that the previous inequalities hold true for all $0<r<\overline r^s_n$.  Therefore, for any fixed $0<r<\overline r^s_n$, the set of initial conditions $\{(s,x_0)\mid \gamma_1^r(s)\le x_0\le\gamma_2^r(s),\, s\in\R\}$ is positively invariant. Moreover, since  $\gamma_1^r(t)\le x^r_-(t)\le\gamma_2^r(t)$ at least for $t<T(n,r,C^s_n)$, then, for all $0<r<\overline r^s_n$, $x^r_-$ must be globally defined and the previous relation of order must hold for all $t\in\R$, which concludes the proof of the first part of (i).\par\smallskip

Concerning the last part of (i), 
let us fix $\ep>0$. From~\eqref{eq:31/03-19:40} and Proposition~\ref{prop:pbk-attr-rep} we have that there is $\beta=\beta(\ep,n,r)<0$ so that
\[
|S^s_n(r,t)-X^s_-|<\frac{\ep}{2}\quad\text{and}\quad |x^r_-(t)-X^s_-|<\frac{\ep}{2}\quad\text{for all }t<\beta(\ep,n,r),
\]
which leads to \eqref{eq:approx-half-lines1}. On the other hand, note that if $r<\min\big\{\overline r^s_n,\sqrt[n+1]{\ep/C^s_n}\big\}$, then~\eqref{eq:approx_are_close} implies that $\beta(\ep,n,r)=\infty$, which ends the proof.
\end{proof}

\begin{rmk}\label{rmk:weakH1}
(i) Let us notice that, in the previous result, the assumption~\ref{H1} can be weakened by considering that the family of quasi-static unstable equilibria $X^u$ is defined only in an interval $[\overline \lambda,\lambda_+]$ with $\lambda_-<\overline \lambda<\lambda_+$ and the equilibrium $X^u(\lambda)$ is hyperbolic for all $\lambda\in(\overline \lambda,\lambda_+]$ but possibly not hyperbolic at $\overline \lambda$. In such a case, however, for every $r>0$ there is $\overline t(r)\in\R$ such that $X^u(rt)$ and consequently all the coefficients $a^u_i$ are defined only for $t\ge \overline t(r)$. Moreover, instead of the inequality on the right-hand side of~\eqref{eq:approx_are_close}, one only attains the inequality on the right-hand side of~\eqref{eq:approx-half-lines} where $\alpha(\ep,n,r)>\overline t(r)$. 

 (ii) The value of $r>0$ up to which the approximations provided in \ref{prop:expansions} are reliable is problem-dependent. For example,  in Figure \ref{fig:2} it is possible to appreciate how the approximations of order $n=1,2,3$ behave with respect to the pullback solutions of a scalar quadratic differential equation with a time-dependent variation of a parameter.
\end{rmk}

Next we aim to use the previous results to provide further information on the occurrence of a rate-induced tipping point. Thanks to Theorem~\ref{thm:fundamentals} we know that,  for $r>0$ sufficiently small, $ x^r_-$ and $x^r_+$ respect the same mutual order as their associated curves of quasi-static equilibria $X^s$ and $X^u$ (wherever they are both defined and comparable), and that  $ x^r_-$ and $x^r_+$ ``collide'' in a rate-induced tipping point if and only if  there is $r^*>0$ such that $  x^{r^{\scaleto{*}{3pt}}}_-= x^{r^{\scaleto{*}{3pt}}}_+$ and $ x^r_-$ loses uniform stability at $r^*$.\par\smallskip

As before, we will simplify the notation by setting 
\[
S^s_n(r,t)=\sum_{i=0}^na_i^s(rt)r^i\quad\text{and}\quad S^u_n(r,t)=\sum_{i=0}^na_i^u(rt)r^i.
\]

Proposition~\ref{prop:expansions} immediately allows to obtain  a simple  sufficient condition for end-point tracking.

\begin{prop}\label{prop:suff-end-point1}
Under the assumptions and notation of {\rm Proposition~\ref{prop:expansions}}, and fixed $\nin$ and $0<r<\overline r_n$, if  there are $\ep>0$ and $T>0$ such that, for all $t>T$,
\begin{equation}\label{eq:27/10-11:24}
X^u_++\ep<S_n^s(r,t)-C_n^s r^{n+1},\quad\text{or }\quad S_n^u(r,-t)+ C_n^ur^{n+1}<X^s_--\ep,\end{equation}
then $x^r_-$ and $x^r_+$ are distinct. In particular, if the future limit-problem $\dot x =f(x)$ does not have any other unstable equilibria, or for any other hyperbolic unstable equilibrium $Y^u_+>X^s_+$ one has that 
\begin{equation}\label{eq:27/10-17:02}
S_n^s(r,t)-C_n^s r^{n+1}<Y^u_+-\ep,
\end{equation}
then $x^r_-$  end-point tracks the associated curve of quasi-static equilibria.
\end{prop}
\begin{proof}
Assume that the first inequality in~\eqref{eq:27/10-11:24} holds true. Since by Proposition~\ref{prop:pbk-attr-rep} we have that for the given $\ep>0$ there is $T_1=T_1(n,\ep)>0$ such that $x^r_+(t)<X^u_++\ep$ for all $t>T_1$, and by Proposition~\ref{prop:expansions} we have that $S_n^s(r,t)-C_n^s r^{n+1}<x^r_-(t)$ for all $t\in\R$, then using the inequality in the hypothesis we immediately obtain that $x^r_+(t)<x^r_-(t)$ for all $t>\max\{T,T_1\}$ which completes the proof of the first statement. Moreover, if  the future limit-problem $\dot x =f(x)$ does not have any other unstable equilibria, this means that $x^r_-$ (which is globally defined and bounded thanks to Proposition~\ref{prop:expansions}) has superior and inferior limits in the basin of attraction of $X^s_+$. Therefore, reasoning as  for~\eqref{eq:15/09-17:00}, we immediately obtain that $x^r_-(t)\to X^s_+$ as $t\to\infty$. The same conclusion holds true if any hyperbolic unstable equilibrium $Y^u_+$ of $Y^u_+$ satisfies~\eqref{eq:27/10-17:02}. Indeed, by reasoning as at the beginning of this proof, one easily obtains that, for $t$ sufficiently big, $x^r_-(t)<y^r_+(t)$, where $y^r_+$ is the locally pullback repelling solution associated to $Y^u_+$ by Proposition~\ref{prop:pbk-attr-rep}. Which again implies that $x^r_-$ end-point tracks the associated curve of quasi-static equilibria.
\end{proof}

\begin{rmk}
The assumptions of Proposition~\ref{prop:suff-end-point1} can be weakened as described in Remark~\ref{rmk:weakH1}. Notice that in this case only the first inequality of the two in~\eqref{eq:27/10-11:24} guarantees the result.
\end{rmk}

Calculating explicitly  $ x^r_-$ and $x^r_+$ is generally unfeasible and their numerical integration is also a far-from-trivial task. However, the asymptotic series approximations obtained in Proposition~\ref{prop:expansions} are easily calculable and  allow us  to characterize the occurrence of  rate-induced tipping point via the change of relative order of solutions lying in a ``close neighborhood'' of $x^r_-$ and $x^r_+$. This idea is inspired, in some sense, by Melnikov's method applied to the non-autonomous perturbation of a homoclinic orbit in a planar system, where the possible crossing between the stable and unstable manifolds is identified by checking the change in relative position between suitable approximations of such manifolds (see Guckenheimer, Holmes~\cite[Section ~4.5]{book:GH}).\par\smallskip 

Before proceeding, let us introduce the definition of visible rate-induced tipping as opposed to what Alkhayoun and Ashwin call invisible tipping \cite{paper:AA}. 

\begin{defn}\label{def:transv_r_tipping}
Let $x^r_-$ be the locally pullback attracting solution and $x^r_+$ be the locally pullback repelling solution for~\eqref{eq:init-prob} respectively associated to the families of quasi-static equilibria $X^s:[\lambda_-,\lambda_+]\to\R$  and $X^u:[\lambda_-,\lambda_+]\to\R$ (see Proposition~\ref{prop:pbk-attr-rep}). We say that $x^r_-$  and $x^r_+$ collide in a \emph{visible rate-induced tipping} at $r=r^*$ if there is $\delta>0$ such that if $X^s(\lambda)>X^u(\lambda)$ \big(resp.~$X^s(\lambda)<X^u(\lambda)$\big) for all $\lambda\in [\lambda_-,\lambda_+]$, then 
\begin{itemize}[leftmargin=*,itemsep=2pt]
\item for all $r\in(r^*-\delta,r^*)$, $x^r_-$ end-point tracks $X^s$ and 
\[
x^r_-(t)>x^r_+(t),\, \big(\text{resp.~} x^r_-(t)<x^r_+(t)\big), \text{ for all }t\in\R;
\]
\item $x^{r^{\scaleto{*}{3pt}}}_-(t)=x^{r^{\scaleto{*}{3pt}}}_+(t)$ for all $t\in\R$;
\item for all $r\in(r^*,r^*+\delta)$,  and $t\in\R$ where $x^r_-$ and $x^r_+$ are both defined,
\[
x^r_-(t)<x^r_+(t),\, \big(\text{resp.~}x^r_-(t)>x^r_+(t)\big).
\]
\end{itemize}
\end{defn}

The main result below offers a necessary and sufficient condition for the occurrence of rate-induced tipping, as well as an upper and lower bound on the tipping value $r^*$ of the parameter,  using only solutions calculated on finite time. This fact has a direct application on the numerical simulation of a tipping phenomenon where the properties of pullback attractivity and pullback repulsivity can not guarantee a reliable finite-time approximation of $x^r_-$ and $x^r_+$ (which are by definition solutions determined by their asymptotic behaviour) after the tipping point.

Despite the relatively technical statement, the idea is simple: Theorem \ref{thm:fundamentals} guarantees that the occurrence of a first tipping point coincides with a collision between the locally pullback solutions $x^r_-$ and $x^r_+$; at the tipping point $x^{r^{\scaleto{*}{3pt}}}_-$ is, by definition, locally pullback attracting and locally pullback repelling; hence, it is possible to find $0<r^*<r_0$ such that,  for all $0<r<r_0$, $x^r_-$ is locally pullback attracting on $\R^-$ and $x^r_+$ is locally pullback repelling on $\R^+$, and one can use the asymptotic series approximation obtained in Theorem \ref{prop:expansions} to choose suitable pairs of solutions (starting at finite time) whose change of relative order at $t=0$ signals the proximity of a rate-induced tipping point. 
For example, consider Figure \ref{fig:3}. The local pullback attractor and repeller are depicted in solid lines (in red the attractor and in blue the repeller) and their first order approximations in dotted lines. Knowing the error of the approximation we can select initial conditions which lie respectively below the attractor and above the repeller (they are depicted in dot-dashed lines). It is possible to appreciate that as $r$ increases and a tipping point approaches, the selected solutions switch relative order at $t=0$---in fact this always happens before the tipping point. Analogously in Figure \ref{fig:4}, the chosen solutions are selected so that one always lies above the attractor and the other below the repeller. A change of relative order at $t=0$ between these two solutions can be related to a change of order at $t=0$ between the pullback solutions. In particular, we may think of the transition at $r=r^*$ in a similar spirit as the occurrence of a heteroclinic connection in a planar  vector field \cite{book:Kuz,paper:WXJ, book:Wig,paper:X}.

\begin{thm}\label{thm:melnikov}
Consider $f:\R\times\R\to\R$ satisfying assumptions~\ref{H0} and~\ref{H1}, and assume that $f$ is smooth and all its partial derivatives of any order with respect to $x$ are bounded on $\big(X^s(\tau),\Lambda(\tau)\big)$ and $\big(X^u(\tau),\Lambda(\tau)\big)$, for all $\tau\in \R$.   
Moreover, fix $\widetilde r>0$ such that for all $0<r\le \widetilde r$, $x^r_-$ is defined and locally pullback attracting in $\R^-$, and $x^r_+$ is defined and locally pullback repelling in $\R^+$. 
 Let $d_0>0$ be the constant given in~\ref{H1},  and consider the functions 
 $D_{out}:\R^+\times\R^+\to\R$ and $D_{in}:\R^+\times\R^+\to\R$ defined by 
\[\begin{split}
D_{out}(\tau,r)=y^r_-(0,-\tau)-y^r_+(0,\tau)&\quad\text{and}\quad
D_{in}(\tau,r)=z^r_-(0,-\tau)- z^r_+(0,\tau),
\end{split}
\]
where,
$y^r_-,y^r_+,z^r_-,z^r_+$ denote the solutions of~\eqref{eq:init-prob},
\[\begin{split}
&y^r_-(t,t_0)=x\big(t,t_0,S_n^s(r,t_0)+\ep\big),\qquad y^r_+(t,t_0)=x\big(t,t_0,S_n^u(r,t_0)-\ep\big)\quad\text{and}\\
 &z^r_-(t,t_0)=x\big(t,t_0,S_n^s(r,t_0)-\ep\big),\qquad z^r_+(t,t_0)=x\big(t,t_0,S_n^u(r,t_0)+\ep\big),
\end{split}
\]
with $\ep\in(0,d_0/2)$, chosen so that  $y^r_-(\cdot,-\tau), z^r_-(\cdot,-\tau)$ are defined in $[-\tau, 0]$ and $y^r_+(\cdot,\tau),z^r_+(\cdot,\tau)$ are defined in $[0,\tau]$  for all $\tau\ge0$ and $0<r<\widetilde r$ (see {\rm Definition~\ref{defn:pullback}}). \\[1ex]
There is $0<\overline r<\widetilde r$ such that $D_{out}(\tau,r)>0$ and $D_{in}(\tau,r)>0$ for all $0<r<\overline r$ and $\tau\ge0$. Moreover, we have that,
\begin{itemize}[leftmargin=20pt, itemsep=2pt]
\item[\rm (a)] there is $r^*\in(0, \widetilde r)$ such that $x^r_-$  and $x^r_+$ collide in a rate-induced tipping at $r=r^*$ if and only if,
\item[\rm(b)] there are constants $r^*\in(\overline r,\widetilde r)$ and $\tau_{in}=\tau_{in}(\ep,n)\ge0$, such that if $\tau>\tau_{in}$, then a $\delta=\delta(n, \tau)>0$ exists such that 
\[
D_{in}\big(\tau,r\big)<0\quad \text{for all }r\in(r^*-\delta, r^*].
\] 
Moreover,  $\delta(n,\tau)$ is decreasing in $\tau$ and  $\lim_{\tau\to\infty}\delta(n,\tau)=0$.
\end{itemize}
Moreover, we have that,
\begin{itemize}[leftmargin=20pt, itemsep=2pt]
\item[\rm (c)]   there are constants  $r^*\in(0, \widetilde r)$ and $\delta>0$ such that $x^r_-(0)\ge x^r_+(0)$ for all $r\in(r^*-\delta,r^*)$, $x^{r^{\scaleto{*}{3pt}}}_-=x^{r^{\scaleto{*}{3pt}}}_+$, and $x^r_-(0)< x^r_+(0)$ for all $r\in(r^*,r^*+\delta)$, if and only if 
\item[\rm (d)]   there are constants $r^*\in(\overline r,\widetilde r)$ and $\delta>0$  such that $D_{out}\big(\tau,r\big)>0$ for all $r\in(r^*-\delta,r^*]$ and $\tau>0$, whereas   for every $r\in(r^*, r^*+\delta)$ there is a  $\tau_{out}=\tau_{out}(\ep,n, r)>0$ so that
\[
D_{out}\big(\tau,r\big)<0\quad\text{for all }\tau>\tau_{out}.
\]
\end{itemize}

 If additionally any point $r^*>0$ for which $x^{r^{\scaleto{*}{3pt}}}_-=x^{r^{\scaleto{*}{3pt}}}_+$ is isolated, then {\rm (d)} implies that $r^*$ is a visible rate-induced tipping point.
\end{thm}

\begin{figure}[htbp]
\centering
\begin{overpic}[trim={1.2cm 0.8cm 0.8cm 0.3cm},clip,width=0.8\textwidth]{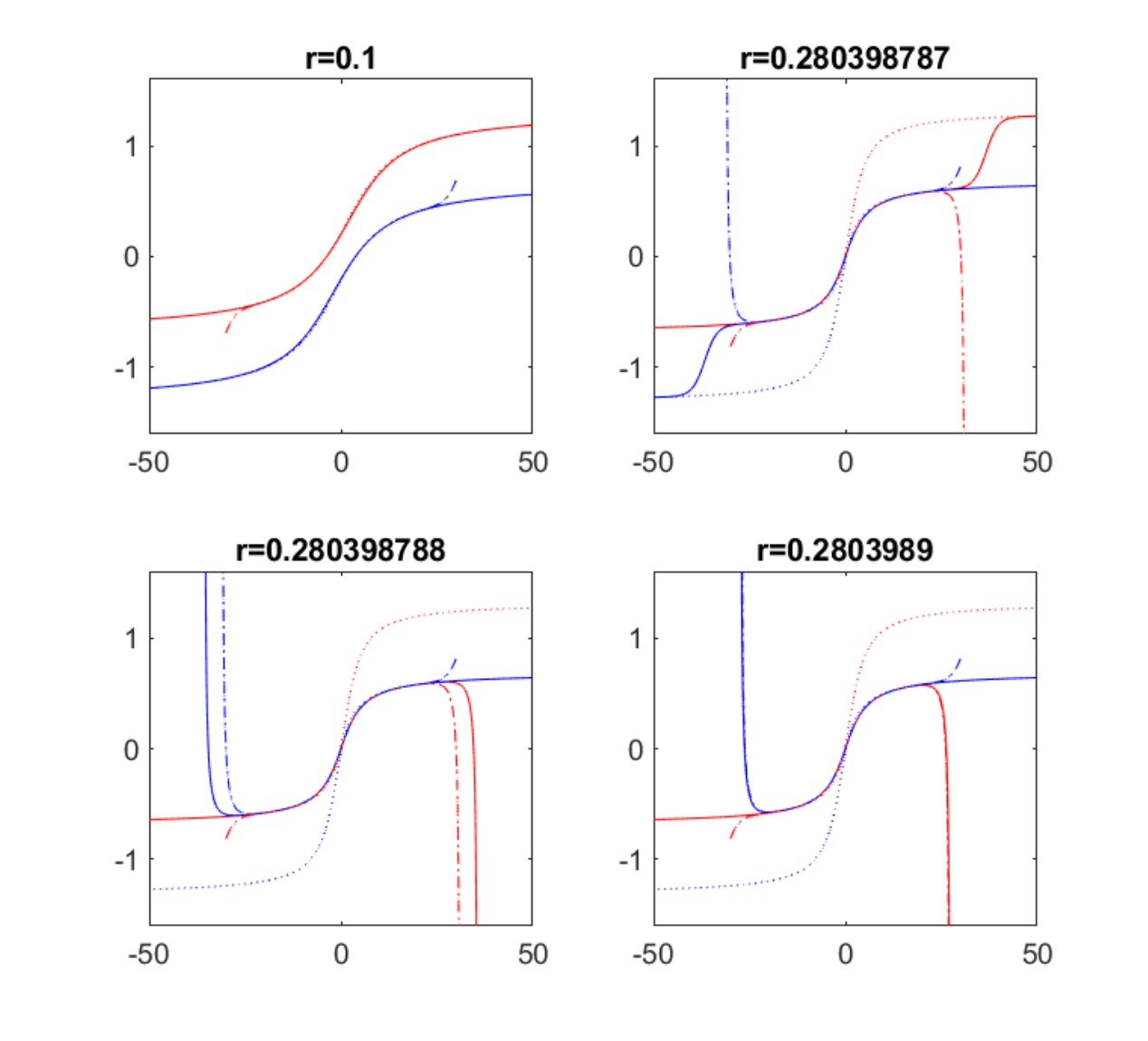}   
\put(36,-1){$t$}	
\put(89,-1){$t$}
\put(-2,32){$x$}
\put(50,32){$x$}
\put(36,50){$t$}	
\put(89,50){$t$}
\put(-2,83){$x$}
\put(50,83){$x$}
\end{overpic}
\caption{\label{fig:3}Numerical simulation of solutions of the scalar differential problem $\dot x= -(x-(2/\pi)\arctan(rt))^2+0.1$ depicting the equivalence (a) $\Leftrightarrow$ (b) in Theorem~\ref{thm:melnikov}. Solid lines represents the locally pullback solutions $x^r_-$ (in red) and $x^r_+$ (in blue). Dotted lines represent the first order approximations of $x^r_-$ (in red) and $x^r_+$ (in blue) as attained from Proposition~\ref{prop:expansions}. Dot-dashed lines represent the solutions $z^r_-(\cdot,-30)$ (in red) and $z^r_+(\cdot, 30)$ (in blue) for $\ep=0.2$. The  top-right panel highlights that the solutions $z^r_-(\cdot,-30)$ and $z^r_+(\cdot, 30)$ blow-up in finite time before the actual rate-induced tipping.} 
\end{figure}
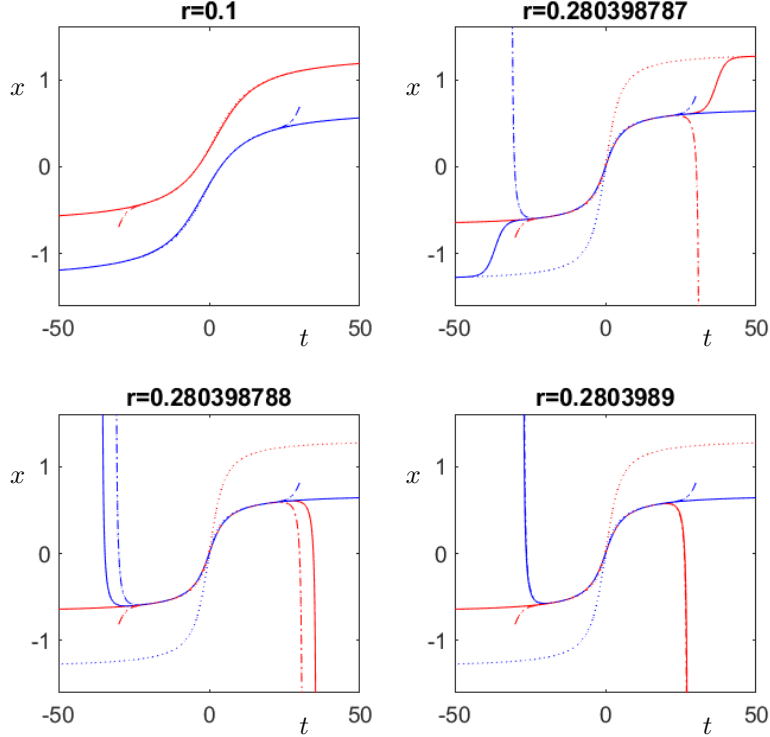

\begin{proof}
First of all notice that, the solutions $y^r_-,y^r_+,z^r_-,z^r_+$ are well-defined thanks to Proposition~\ref{prop:expansions} and  Definition~\ref{defn:pullback}. Moreover, due to Propositions~\ref{prop:hyperbolic} and~\ref{prop:expansions}, and since $\ep\in(0,d_0/2)$,  there is $\overline r>0$ such that if $0<r<\overline r$, then  $D_{out}(\tau,r)$ and $D_{in}(\tau,r)$ are strictly greater than zero for all $\tau\ge0$. Next we prove that (a) $\Leftrightarrow$ (b).\par\smallskip

(a) $\Rightarrow$ (b). Firstly, let us consider $r=r^*$. Due to Theorem~\ref{thm:fundamentals}(iii), we have that $x^{r^{\scaleto{*}{3pt}}}_-=x^{r^{\scaleto{*}{3pt}}}_+$ and $\lim_{t\to\infty} x^{r^{\scaleto{*}{3pt}}}_-(t)=X^u_+$. Therefore, and thanks also to  Proposition~\ref{prop:expansions}, there is $\tau_{in}=\tau_{in}(\ep,n)\ge0$, such that for all $\tau>\tau_{in}$
\[
S_n^s(r^*,-\tau)-\ep<x^{r^{\scaleto{*}{3pt}}}_-(-\tau),\quad\text{and}\quad x^{r^{\scaleto{*}{3pt}}}_-(\tau)<S_n^u(r^*,\tau)+\ep,
\] 
which implies 
\begin{equation}\label{eq:20/08-12:37}
z^{r^{\scaleto{*}{3pt}}}_-(0,-\tau)<x^{r^{\scaleto{*}{3pt}}}_-(0)<z^{r^{\scaleto{*}{3pt}}}_+(0,\tau)\quad\text{for all }\tau>\tau_{in}.
\end{equation}
Hence, $D_{in}\big(\tau,r^*\big)<0$ for all $\tau> \tau_{in}$. \par\smallskip

Now, from the continuity of $S_n^u$ and $S_n^s$,  for every fixed $\mu>0$ and  $\tau>0$ a $\delta=\delta(n,\mu,\tau)>0$ exists  such that for all $r\in(r^*-\delta,r^*)$,
\[
|S_n^u(r^*,\tau) - S_n^u(r,\tau)|<\mu\quad\text{and}\quad
|S_n^s(r^*,-\tau)-S_n^s(r,-\tau)|<\mu.
\]
Therefore, from Lemma~\ref{lem:kamke}, for every $\tau>0$ there is $\delta=\delta(n,\tau)>0$ such that if $r\in(r^*-\delta,r^*)$, then
\begin{equation}\label{eq:23/10-15:27}
\begin{split}
\|z^{r^{\scaleto{*}{3pt}}}_-(\cdot,-\tau)-z^r_-(\cdot, -\tau)\|_{\mathcal{C}([-\tau, 0],\R)}<\frac{\big|D_{in}\big(\tau,r^*\big)\big|}{2}\quad\text{and}&\\
\|z^{r^{\scaleto{*}{3pt}}}_+(\cdot,\tau)-z^r_+(\cdot,\tau)\|_{\mathcal{C}([ 0,\tau],\R)}<\frac{\big|D_{in}\big(\tau,r^*\big)\big|}{2}.&
\end{split}\end{equation}
Hence, gathering together \eqref{eq:20/08-12:37} and \eqref{eq:23/10-15:27}, we obtain that for every $\tau>\tau_{in}$ there is $\delta(n,\tau)>0$ such that for all $r\in(r^*-\delta,r^*)$,
\[
D_{in}\big(\tau,r\big)=z^r_-(0,-\tau)-z^r_+(0,\tau)
<0,
\]
which is the aimed inequality.
Finally, considered $\tau_{in}<\tau_1<\tau_2$, from~\eqref{eq:23/10-15:27} we easily have that $0<\delta(n,\tau_2)\le \delta(n,\tau_1)<r^*$. Therefore, the function $\delta_n: (0,\infty)\to\R^+$, $\tau\mapsto \delta(n,\tau)$ has limit $\overline \delta\ge0$. We shall prove that $\overline \delta=0$. Assume by contradiction that $\overline \delta>0$ and consider $r\in(r^*-\overline \delta,r^*)$. Then we would have that 
\begin{equation}\label{eq:24/10-20:11}
D_{in}\big(\tau,r\big)<0,\quad\text{for all }\tau>\tau_{in}.
\end{equation}
Now, by assumption,  $x^r_-(t)\to X^s_+$ and $x^r_+(-t)\to X^u_+$ as $t\to\infty$. However, this implies that there is $\overline\tau_1 \ge\tau_{in}$ such that, for all $\tau>\overline \tau_1$,  $x^r_+(-\tau)<S_n^s(r,- \tau)-\ep$. Therefore, from Proposition~\ref{prop:go-up} we have that, for all $\tau>\overline \tau_1$, $z^r_-(t,-\tau)$ is defined for all $t>-\tau$ and $z^r_-(t,-\tau)\to X^s_+$ as $t\to\infty$. In particular, due to the uniform asymptotic stability of $x^r_-$ and Proposition~\ref{prop:go-up}, there is $\overline \tau_2>0$ such that for all $t>\overline \tau_2$ one has that 
\[
0<x^r_-(t)-z^r_-(t,-\tau)<x^r_-(t)-S_n^u(r,t)+\ep
\]
Thus, for all $\tau>\max\{\overline \tau_1,\overline \tau_2\}$ we have that 
\[
D_{in}\big(\tau,r\big)=z^r_-(0,-\tau)-z^r_+(0,\tau)>0,
\]
which contradicts~\eqref{eq:24/10-20:11}. Hence it must be 
\[
\lim_{\tau\to\infty}\delta_n(\tau)=0,
\]
as claimed.
\par\smallskip

(b) $\Rightarrow$ (a). 
Consider a sequence $(\tau_k)_{k\in\N}$ such that \[D_{in}\big(\tau_k,r^*\big)=z^{r^{\scaleto{*}{3pt}}}_-(0,-\tau_k)- z^{r^{\scaleto{*}{3pt}}}_+(0,\tau_k)<0,\quad \text{for all }k\in\N.
\]
Such a sequence exists due to (b). Note also that since by assumptions $x^{r^{\scaleto{*}{3pt}}}_-$ is locally pullback attracting on $\R^-$, and $x^{r^{\scaleto{*}{3pt}}}_+$ is locally pullback repelling on $\R^+$, when we take the limit as $k\to\infty$ in the previous formula, we obtain that,
\[
x^{r^{\scaleto{*}{3pt}}}_-(0)-x^{r^{\scaleto{*}{3pt}}}_+(0)\le 0.
\]
In order to prove that the previous inequality is, in fact, an equality, let us assume by contradiction that $x^{r^{\scaleto{*}{3pt}}}_-(0)-x^{r^{\scaleto{*}{3pt}}}_+(0)< 0$. Then, by continuous variation of the solutions, there is $0<\rho^*<r^*$ such that $x^{\rho^{\scaleto{*}{3pt}}}_-(0)=x^{\rho^{\scaleto{*}{3pt}}}_+(0)$, i.e.~$\rho^*$ is a tipping point. From (a) $\Rightarrow$ (b) proved above, we have that $D_{in}\big(\tau,\rho^*\big)<0$ for all $\tau> \tau_{in}$. However, note also that, since $\rho^*<r^*$, from (b), there is $\tau_{\rho^{\scaleto{*}{3pt}}}>0$ such that $D_{in}\big(\tau_{\rho^{\scaleto{*}{3pt}}},\rho^*\big)>0$, but this is clearly a contradiction. Therefore, $x^{r^{\scaleto{*}{3pt}}}_-(0)=x^{r^{\scaleto{*}{3pt}}}_+(0)$ which concludes the proof of this implication.
\par\smallskip

\begin{figure}[htbp]
\centering
\begin{overpic}[trim={1.2cm 0.8cm 0.8cm 0.3cm},clip,width=0.8\textwidth]{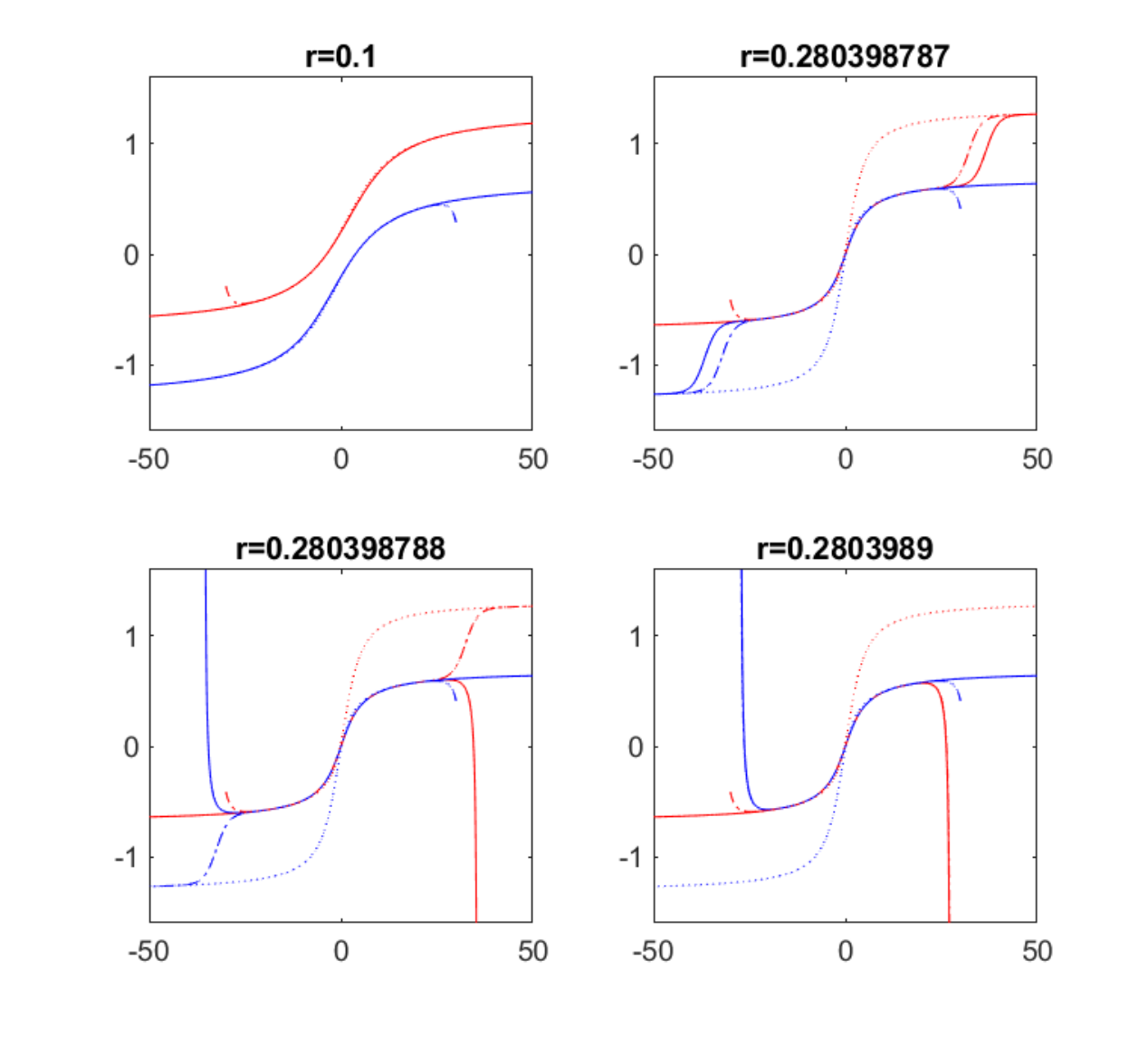}   
\put(36,-1){$t$}	
\put(89,-1){$t$}
\put(-2,32){$x$}
\put(50,32){$x$}
\put(36,50){$t$}	
\put(89,50){$t$}
\put(-2,83){$x$}
\put(50,83){$x$}
\end{overpic}
\caption{\label{fig:4}Numerical simulation of solutions of the scalar differential problem $\dot x= -(x-(2/\pi)\arctan(rt))^2+0.1$ depicting the equivalence (c) $\Leftrightarrow$ (d) in Theorem~\ref{thm:melnikov}. Solid lines represents the locally pullback solutions $x^r_-$ (in red) and $x^r_+$ (in blue). Dotted lines represent the first order approximations of $x^r_-$ (in red) and $x^r_+$ (in blue) as attained from Proposition~\ref{prop:expansions}. Dot-dashed lines represent the solutions $y^r_-(\cdot,-30)$ (in red) and $y^r_+(\cdot, 30)$ (in blue) for $\ep=0.2$. The comparison between the bottom panels highlights that the solutions  $y^r_-(\cdot,-30)$  and $y^r_+(\cdot, 30)$  blow-up in finite time only after the actual rate-induced tipping.} 
\end{figure}

\begin{figure}[htbp]
\centering
\begin{overpic}[trim={0cm 8.7cm 0cm 0cm},clip,width=\textwidth]{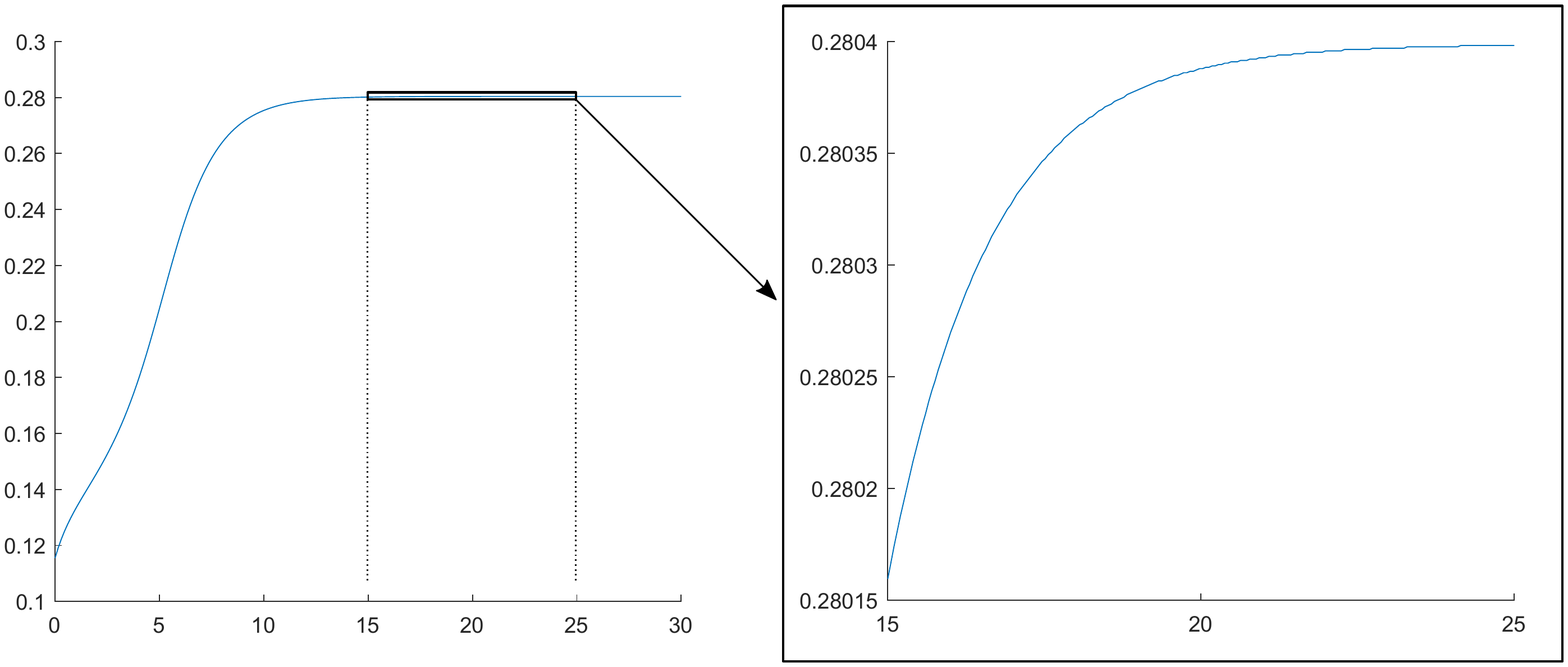}   
\put(92,1.5){$\tau$}	
\put(41,1.5){$\tau$}
\put(8,36){$r^*-\delta(1,\tau)$}
\end{overpic}
\caption{\label{fig:5} 
Numerical simulation of $r^*-\delta(1,\tau)$ as a function of $\tau\in[0,30]$ on the left-hand side and magnification on the right-hand side, for the scalar differential problem $\dot x= -(x-(2/\pi)\arctan(rt))^2+0.1$. The convergence of $\delta$ to zero as proved in (a) $\Rightarrow$ (b) of Theorem \ref{thm:melnikov} becomes apparent as $\tau$ increases.}
\end{figure}

(c) $\Rightarrow$ (d). 
Consider  $r^*\in(0,\widetilde r)$  $\delta>0$ as in (c). Then, $x^r_-(0)\ge x^r_+(0)$ for all $r\in(r^*-\delta,r^*)$
and $t\in\R$. By construction, we immediately have that $y^r_-(0,-\tau)> y^r_+(0,\tau)$, for all  $r\in(r^*-\delta,r^*)$ and $\tau>0$,
which gives us the first property in (d).

Now, consider $r\in(r^*,r^*+\delta)$. From the assumption in (c), we have that $0< x^r_+(0)-x^r_-(0)=:d_1(r)$. In particular, thanks to Proposition~\ref{prop:expansions},  the local pullback attractivity of $x^r_-$ and the local pullback repulsivity of $x^r_+$ (see Definition~\ref{defn:pullback}), there is a $\tau_{out}=\tau_{out}(\ep,n, r)>0$ such that for all $\tau>\tau_{out}$,
\[
\begin{split}
&0<y^r_-(0,-\tau) -x^r_-(0)<\frac{d_1(r)}{2}\quad\text{and}\quad 0<x^r_+(0)-y^r_+(0,\tau)<\frac{d_1(r)}{2}
\end{split}
\]
Consequently, $D_{out}(\tau,r)<0$ for all $\tau>\tau_{out}$.\par\smallskip

(d) $\Rightarrow$ (c)
 Consider sequences $(r_k)_{k\in\N}$ and $(\tau_k)_{k\in\N}$ such that $r_k\searrow r^*$ and for all $k\in\N$, $\tau_{k+1}>\max\{\tau_k,\tau_{out}(\ep,n, r_{k+1})\}$. Then, we have that for all $k\in\N$,
\[
D_{out}(\tau_k,r_k)=y^{r_k}_-(0,-\tau_k)-y^{r_k}_+(0,\tau_k)<0.
\]
In particular, since by definition $x^r_-(0)<y^r_-(0,-\tau)$ and $y^r_+(0,\tau)<x^r_+(0)$ for all $r>0$ and $\tau>0$, then 
\[
x^{r_k}_-(0)-x^{r_k}_+(0)<0,\quad\text{for all }k\in\N.
\]
Taking the limit as $k\to\infty$ we have that 
$x^{r^{\scaleto{*}{3pt}}}_-(0)-x^{r^{\scaleto{*}{3pt}}}_+(0)\le0$. Let us assume that $x^{r^{\scaleto{*}{3pt}}}_-(0)-x^{r^{\scaleto{*}{3pt}}}_+(0)<0$ and prove by contradiction that, in fact, the equality must hold. If the previous inequality is strict, there is $0<\rho^*<r^*$ such that $x^{\rho^{\scaleto{*}{3pt}}}_-(0)=x^{\rho^{\scaleto{*}{3pt}}}_+(0)$ and there are no further tipping points in the interval $(\rho^*, r^*)$. In particular, for all $r\in(\rho^*, r^*]$, $x^r_-(0)<x^r_+(0)$. Therefore, from the second part of the proof of the implication (c) $\Rightarrow$ (d), now applied to $\rho^*$, for every $r\in(r^*-\delta,r^*)$ there is $\tau>0$ such that $D_{out}\big(\tau,r\big)<0$. However, this is in contradiction with the assumption in (c). Hence, it must be  $x^{r^{\scaleto{*}{3pt}}}_-=x^{r^{\scaleto{*}{3pt}}}_+$. Note also that, since by definition $x^r_-(0)<y^r_-(0,-\tau)$ and $y^r_+(0,\tau)<x^r_+(0)$ for all $r>0$ and $\tau>0$ and, by assumption, for every $r\in(r^*, r^*+\delta)$ there is a  $\tau_{out}=\tau_{out}(\ep,n, r)>0$ so that $D_{out}\big(\tau,r\big)<0$ for all  $\tau>\tau_{out}$, then it must be that $x^r_-(0)<x^r_+(0)$ for all $r\in(r^*, r^*+\delta)$.
Finally, for all $r\in(r^*-\delta, r^*)$ it must be $x^r_-(0)\ge x^r_+(0)$. Otherwise, for each $r\in(r^*-\delta,r^*)$ where this is not true there would be $\tau>0$ such that $D_{out}\big(\tau,r\big)<0$ and this would contradict (d). In particular, note that  if we assume that $x^{r}_-\neq x^{r}_+$ for all $r\in (r^*-\delta,r^*)\cup(r^*, r^*+\delta)$ then the same reasoning would give us that $x^r_-(0)> x^r_+(0)$ for all $r\in(r^*-\delta, r^*)$, which, together the previous part, implies that $r^*$ is a visible rate-induced tipping point.
\end{proof}

For an illustration of the loss of uniform asymptotic stability as well as the approximation results by asymptotic series for the occurrence of a rate-induced tipping point, we refer to Figures~\ref{fig:2}-\ref{fig:5}, where we consider the example
\begin{equation*}
\dot x= -(x-(2/\pi)\arctan(rt))^2+\zeta
\end{equation*}
for constants $\zeta=1.1$ and $\zeta=0.1$.

\subsection*{Advantages and limitations}
Hereby, we wish to briefly comment on the applicability, reliability and limitations of the methods and results contained in this section. As we have noted, the occurrence of rate-induced tipping is strictly connected to the asymptotic behaviour of locally pullback attracting and repelling solutions. The very nature of pullback solutions make them (in general) hard to be calculated explicitly and also integrated numerically. Consequently, the reliability of numerical simulations must be proved example by example---a finite time integration does not necessarily guarantee a trustworthy approximation. On the other hand, the construction of asymptotic series expansions---as in Proposition \ref{prop:expansions}---is generic and can be carried out algorithmically using only values of the vector field and its derivatives provided that the appropriate assumptions are satisfied. Therefore, Proposition \ref{prop:suff-end-point1} and Theorem \ref{thm:melnikov} provide rigorous characterizations of end-point tracking and rate-induced tipping that rely only on calculable quantities. Particularly, Theorem \ref{thm:melnikov} shows a constructive rigorous way of running reliable numerical simulations of a rate-induced tipping event for scalar differential equations with a time-dependent drift of a parameter. Figure \ref{fig:5} highlights how the convergence towards the expected tipping point is obtained for values of $\tau$ relatively small, especially if compared to the initial conditions that might be required to reliably approximate the pullback solutions. We are not aware of any other technique in the literature that is at the same time generic, rigorously proved and constructive not even for scalar problems.

A natural question attains the applicability of these methods to higher-dimensional problems. A thorough look at the proof of Proposition \ref{prop:expansions} shows that, when curves of quasi-static equilibria are involved and the required assumptions of regularity are satisfied, the construction of the asymptotic series approximations of the locally pullback solutions is possible also when $N\ge1$. It is clear, however, that both Proposition \ref{prop:suff-end-point1} and Theorem \ref{thm:melnikov} require a well-defined relation of order on the phase space. This fact restricts the applicability of the obtained result to those cases where higher-dimensional system can be brought back to the analysis of a scalar problem---for example through a center manifold reduction method---or where a certain property of monotonicity of the flow is in force. Yet, since we have proven that rate-induced tipping points can be related to a merging of global solutions, the dynamical situation is similar to global homoclinic/heteroclinic  bifurcations \cite{book:Kuz,paper:WXJ, book:Wig,paper:X}. For these global bifurcations, center manifold results are already available~\cite{paper:Sandstede}, so we expect that suitable, yet non-trivial, modifications of our methodology do apply in higher dimensions. We leave these and other possible extensions open for future work. 

\subsection*{Acknowledgments}
We thank two anonymous referees for their comments and suggestions, which have been  included in the current version of the paper.


\begin{thebibliography}{99}

\bibitem{paper:AA} \textsc{H.M.~Alkhayoun, P.~Ashwin}:
        Rate-induced tipping from periodic attractors: partial
        tipping and connecting orbits,
        {\em Chaos\/} {\bf 28} (3) (2018), 033608, 11 pp.
        
\bibitem{paper:AAJ} \textsc{H.M.~Alkhayoun, P.~Ashwin, L.C.~Jackson, C.~Quinn, R.A.~Wood}:
        Basin bifurcations, oscillatory instability and rate-induced thresholds for Atlantic meridional overturning circulation in a global oceanic box model
        \emph{Proc. R. Soc. Lond. A} {\bf 475} (2225) (2019), 20190051.        
              
        \bibitem{book:AFP} \textsc{L. Ambrosio, N. Fusco, D. Pallara}: \emph{Functions Of Bounded Variation And Free Discontinuity Problems}, \textrm{Oxford Mathematical Monographs. The Clarendon Press Oxford University Press, New York, 2000.}

\bibitem{paper:Art} \textsc{Z. Artstein}: Uniform asymptotic stability via the limiting equations, \emph{J. Differential Equations}, \textbf{27} (1978), 172--189.
        
\bibitem{paper:APW} \textsc{P. Ashwin, C. Perryman, S.~Wieczorek}: Parameter shifts for nonautonomous systems in low dimension: bifurcation- and rate-induced tipping, \emph{Nonlinearity} \textbf{30} (2017), 2185--2210.

\bibitem{paper:AWVC} \textsc{P. Ashwin,  S.~Wieczorek, R. Vitolo, P. Cox}: Tipping points in open systems: bifurcation,
noise-induced and rate-dependent examples in the climate system, \emph{Phil. Trans. R. Soc. A } \textbf{370} (2012), 1166--1184.

\bibitem{paper:Ca} \textsc{G. Carigi}: Rate-induced tipping in nonautonomous dynamical systems with bounded noise,
MRes Thesis,  University of Reading, 2017.

\bibitem{paper:CL} \textsc{A.N. Carvalho, J.A. Langa}: Non-autonomous perturbation of autonomous semilinear differential equations: continuity of local stable and unstable manifolds, \emph{J. Differential Equations}, \textbf{233} (2007), 622--653.

\bibitem{paper:CLRS} \textsc{A.N. Carvalho, J.A. Langa, J.C. Robinson and A. Suarez}: Characterization of non-autonomous attractors, \emph{Cadernos de Matematica}, \textbf{07} (2006), 277--302.


\bibitem{book:E} \textsc{J. Eldering}: \emph{Normally hyperbolic invariant manifolds - the noncompact case}. \textrm{Springer, Atlantis Series in Dynamical Systems vol 2, 2013.}

\bibitem{book:GH} \textsc{J. Guckenheimer, P. Holmes}: \emph{Nonlinear Oscillations, Dynamical Systems, and Bifurcations of Vector Fields}, corrected 7th printing. \textrm{Springer, 2002.}

\bibitem{book:HALE} \textsc{J.K.~Hale}: \emph{Ordinary Differential Equations}, \textrm{Wiley-Interscience, New York, 1969.}

\bibitem{paper:Hartl} \textsc{M.~Hartl}: Non-autonomous Random Dynamical Systems: Stochastic Approximation and
Rate-Induced Tipping,  PhD Thesis, Imperial College London, 2019.

\bibitem{paper:Hill} \textsc{A.V. Hill}: Excitation and accommodation in nerve, \emph{Proc. R. Soc. Lond. B} \textbf{119} (1936),  305--355.



\bibitem{paper:HNRS}  \textsc{A.~Hoyer-Leitzel, A.~Nadeau, A.~Roberts,  A.~Steyer}: Detecting transient rate-tipping using Steklov averages and Lyapunov vectors.  \emph{arXiv preprint arXiv:1702.02955} (2017).

\bibitem{paper:Kato} \textsc{T.~Kato}: On the Adiabatic Theorem of Quantum Mechanics,  \emph{J. Phys. Soc. Japan}  \textbf{5} (6) (1950),  435--439.

\bibitem{paper:Kiers} \textsc{C.~Kiers}: Rate-Induced Tipping in Discrete-Time Dynamical Systems,  \emph{SIAM J. Appl. Dyn. Syst.}  \textbf{19} (2) (2020), 1200--1224.

\bibitem{paper:KJ} \textsc{C.~Kiers, C.K.R.T.~Jones}: On Conditions for Rate-induced Tipping in Multi-dimensional Dynamical Systems,  \emph{J.~Dynam.~Differential Equations}  \textbf{32} (2020), 483--503.

\bibitem{book:KR} \textsc{P.~Kloeden, M.~Rassmussen},
        \emph{Nonautonomous Dynamical Systems},
       \textrm{ Mathematical Surveys and Monographs,
        Amer.~Math.~Soc., 2011.}

\bibitem{book:K} \textsc{C.~Kuehn}: \emph{Multiple Time Scale Dynamics}, volume 191. \textrm{Springer, 2015.}

\bibitem{paper:KCT2} \textsc{C.~Kuehn}: A mathematical framework for critical transitions: normal forms, variance and applications, \emph{J. Nonlinear Sci.} \textbf{23}(3) (2013), 457--510.

\bibitem{book:Kuz} \textsc{Y.A.~Kuznetsov}: \emph{Elements of applied bifurcation theory}, volume 112. \textrm{Springer Science \& Business Media, 2013.}


\bibitem{paper:LHK} \textsc{T.M.~Lenton, H.~Held, E.~Kriegler, J.W.~Hall, W.~Lucht, S.~Rahmstorf, H.J.~Schellenhuber}: Tipping elements in the earth's climate system, \emph{Proc.~Natl Acad.~Sci.~} \textbf{105} (2008), 1786--1793.


\bibitem{paper:LNOR} \textsc{I.P.~Longo, C.~N\'{u}\~{n}ez,~R. Obaya, M.~Rasmussen}: Rate-induced tipping and saddle-node bifurcation for quadratic differential equations with nonautonomous limiting equations, \emph{Preprint} (2020).

\bibitem{paper:MLS} \textsc{R.M. May, S.A. Levin, and G. Sugihara}: Complex systems: Ecology for bankers, \emph{Nature} \textbf{451} (2008), 893--895.

\bibitem{paper:Meisel} \textsc{C.~Meisel, A.~Klaus, C.~Kuehn, and D.~Plenz}: Critical slowing down at saddle-node bifurcation controls pyramidal neuron spiking, \emph{PLoS Comp. Biol.} \textbf{11} (4) (2015), e1004097.

\bibitem{paper:OKW}  \textsc{P.E.~O'Keeffe, S.~Wieczorek}: Tipping phenomena and points of no return in ecosystems: beyond classical bifurcations, \emph{ SIAM Journal on Applied Dynamical Systems} \textbf{19} (4) (2020), 2371--2402.


	
	
	
\bibitem{paper:PW} \textsc{C.~Perryman, S.~Wieczorek}: Adapting to a changing environment:
non-obvious thresholds in multi-scale systems, \emph{Proc.~R.~Soc.~A} \textbf{470}: 20140226 (2014).
\bibitem{potz} \textsc{C.~P\"{o}tzsche}:
        Nonautonomous continuation of bounded solutions,
        {\em Commun. Pure Appl. Anal.} {\bf 10} (3) (2011), 937--961.

\bibitem{rasmln} \textsc{M.~Rasmussen}:
           {\em Attractivity and bifurcation for nonautonomous dynamical systems},
     Lecture Notes in Math. {\bf 1907}, Springer-Verlag, Berlin, 2007.

\bibitem{paper:Sandstede} \textsc{B. Sandstede}:  Center manifolds for homoclinic solutions, \emph{Journal of Dynamics and Differential Equations} \textbf{12}(3) (2000), 449--510.

\bibitem{scheffer} \textsc{M.~Scheffer}:
        {\em Critical Transitions in Nature and Society},
        \textrm{Princeton University Press, 2009.}

\bibitem{paper:SvNHH} \textsc{M. Scheffer, E.H. van Nes, M. Holmgren and T. Hughes}: Pulse-driven loss of top-down control: the critical-rate hypothesis, \emph{Ecosystems} \textbf{11} (2008), 226--237.

\bibitem{paper:Sell} \textsc{G.R.~Sell}: Nonautonomous differential equations and topological dynamics I and II, \emph{Trans.~Am.~Math.~Soc.~}\textbf{127} (1967), 241--283.
\bibitem{book:GS} \textsc{G.R.~Sell}: \emph{Topological Dynamics and Ordinary Differential Equations}, \textrm{Van Nostrand-Reinhold, London, 1971.}

\bibitem{paper:VWF} \textsc{A.~Vanselow, S.~Wieczorek, U.~Feudel}: When very slow is too fast-collapse of a predator-prey system, \emph{J. Theor. Biol.} \textbf{479} (2019), 64--72.

\bibitem{paper:WALC} \textsc{S.~Wieczorek, P.~Ashwin, C.M.~Luke, P.M.~Cox}: Excitability in ramped systems:
the compost-bomb instability, \emph{Proc. R. Soc. A} \textbf{467} (2010), 1243--1269.

\bibitem{paper:WXJ} \textsc{S.~Wieczorek, C.~Xie, C.K.R.T.~Jones}: Compactification for asymptotically autonomous dynamical systems: theory, applications and invariant manifolds,  \emph{Nonlinearity}, \textbf{34} (5) (2021), 2970.

\bibitem{book:Wig} \textsc{S.~Wiggins}:
        {\em Global bifurcations and chaos: analytical methods}, Vol. 73
        \textrm{Springer Science \& Business Media, 2013.}

\bibitem{paper:X} \textsc{C.~Xie}: Rate-Induced Critical Transitions, PhD Thesis, University College Cork, 2020.

\bibitem{paper:YSY} \textsc{V. Yukalov, D. Sornette, E. Yukalova}: Nonlinear dynamical
model of regime switching between conventions and business cycles,  \emph{J. Econ. Behav. Organ.} \textbf{70} (2009), 206--230.

\end{thebibliography}
\end{document}